%% file: ddgraphs.tex
\documentclass[11pt]{article}

\usepackage{authblk}
\usepackage[utf8]{inputenc}
\usepackage{graphicx} 
\usepackage{booktabs}
\usepackage{enumerate}
\usepackage{array}
\usepackage{subfig}
\usepackage{amsmath}
\usepackage{amssymb}
\usepackage{longtable}
\usepackage{lscape}

\usepackage{tikz}
\usetikzlibrary{calc}
\tikzstyle{vertex}=[circle,fill=black]
\tikzstyle{point}=[circle,fill=black]
\tikzstyle{edge} = [draw,thick,-]
\tikzstyle{line} = [draw,thick,-]

\newcommand{\cth}[1]{\multicolumn{1}{c}{#1}}

\newcommand{\QED}{\hfill$\blacksquare$\\}

\newtheorem{theorem}{Theorem $\!\!$}[section]
\newtheorem{lemma}[theorem]{Lemma $\!\!$}

\newtheorem{definition}[theorem]{Definition $\!\!$}
\newenvironment{proof}[1][Proof]%
{\begin{description}\item[\noindent\textbf{#1}:]}
{\QED\end{description}}

\newcommand{\mathCQ}{$C_4^q$}
\newcommand{\CQ}{\protect\mathCQ-factor }
\newcommand{\CQs}{\protect\mathCQ-factors }
\newcommand{\CQp}{\protect\mathCQ-factor}
\newcommand{\CQsp}{\protect\mathCQ-factors}

\newcommand{\CQmarked}{\protect\mathCQ-marked }

\newcommand{\CQmarkable}{\protect\mathCQ-markable }

\newcommand{\CQmarkedpg}{\CQmarked pregraph }

\newcommand{\pCQmarkedpg}{partial \CQmarked pregraph }
\newcommand{\pCQmarkedpgp}{partial \CQmarked pregraph}
\newcommand{\pCQmarkedpgs}{partial \CQmarked pregraphs }
\newcommand{\pCQmarkedpgsp}{partial \CQmarked pregraphs}

\title{Generation of Delaney-Dress graphs}
\author{Nico Van Cleemput}
\affil{Department Applied Mathematics, Computer Science and Statistics\\Ghent University\\Krijgslaan 281 - S9 - WE02\\9000 Ghent\\Belgium\\\texttt{nico.vancleemput@gmail.com}}
\begin{document}

\maketitle

\begin{abstract}
We introduce an algorithm for the efficient generation of cubic pregraphs which have a 2-factor in which each component is a quotient of $C_4$. This class of pregraphs is of particular interest, since it corresponds to the class of uncoloured graphs that are the underlying graphs of Delaney-Dress graphs. We also extend the algorithm to generate Delaney-Dress graphs.
\end{abstract}

\begin{description}
\item[Keywords] Delaney-Dress graph; cubic pregraph; symmetry type graph; isomorphism-free, exhaustive generation; closed structure
\end{description}

%%%%%%%%%%%%%%%%%%%%%%%%%%%%%%%%%%%%%%%%%%%%%%%%%

\section{Introduction}

A \emph{pregraph} is a multigraph which can also contain loops and semi-edges. In this paper we will only consider connected pregraphs. A \emph{quotient of $C_4$} is a graph isomorphic to one of the four graphs in Figure~\ref{fig:c4-quotients}. A \emph{\CQp} is a 2-factor in which each component is a quotient of $C_4$. A \emph{Delaney-Dress graph} is a 3-edge-coloured cubic pregraph in which the components with colours 0 and 2 form a \CQp.
These graphs arise during the study of periodic tilings. In the context of maps, i.e., 2-cell embeddings of a connected graph in a closed surface, they are also called \emph{symmetry type graphs}.
They are obtained by the taking the quotient of the dual of the barycentric subdivision of the tiling or the graph with the colour preserving automorphism group of that dual. See \cite{DrBr96,pregraphs,HOPF:13} for more details.

\begin{figure}[h]
\begin{center}
\input{q-1.tex}
\input{q-2.tex}
\input{q-3.tex}
\input{q-4.tex}
\end{center}

\caption{The four types of quotients of $C_4$. From left to right: $q_1$, $q_2$, $q_3$ and $q_4$.}\label{fig:c4-quotients}
\end{figure}
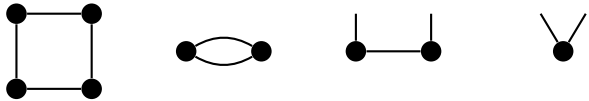

In \cite{pregraphs} an algorithm to generate different classes of cubic pregraphs is presented.
That article also describes a filtering algorithm which can decide in linear time whether a pregraph has a \CQp.
This class of pregraphs is of particular interest, since it corresponds to the class of uncoloured graphs that are the underlying graphs of Delaney-Dress graphs.

In \cite{pregraphs}, it is also noted that, although the filtering algorithm is efficient, a specialised algorithm for that class should be developed, since only a very small proportion of 3-edge-colourable pregraphs, i.e., the class most suited to filter the symmetry type graphs from, has this property.
Once we reach 20 vertices already 99.98\% of the graphs that are generated do not have a \CQ (see Table~\ref{tab:ddgraphratio}).
Section~\ref{sec:cq4marked} introduces such an algorithm. We also use this algorithm to go a step further, and also generate Delaney-Dress graphs.

Delaney-Dress graphs are of interest, since they encode information about the symmetries of a map. From the Delaney-Dress graph of a map, we can see whether the map is regular, transitive on either vertices, edges or faces, as well as several other properties \cite{HOPF:13}. Knowing which Delaney-Dress graphs exist, is therefore crucial when classifying maps. Delaney-Dress graphs also form the underlying structures for Delaney-Dress symbols which fully encode tilings together with their symmetry group. There are several examples of ad-hoc algorithms to generate specific types of Delaney-Dress symbols with applications in mathematics \cite{DressScharlau86,DressHuson87} and in chemistry \cite{DrBr96,azul}.  An algorithm to generate Delaney-Dress graphs forms the basis for a general approach to this type of problems.

\begin{table}
\begin{center}
\begin{tabular}{rrrr@{.}l} 
\toprule
$n$&  {\hfill colourable \hfill} &  {\hfill has \CQ \hfill}  &  \multicolumn{2}{c}{ratio} \\ 
\midrule
1 &  1 & 1 &  100&00 \%\\ 
2 &  3 &  3 &  100&00 \%\\ 
3 &  3 &  2 &  66&67\%\\ 
4 &  11 &  9 &  81&82\%\\ 
5 &  17 &  7 &  41&18\%\\ 
6 &  59 &  29 &  49&15\%\\ 
7 &  134 &  27 &  20&15\%\\ 
8 & 462 &  105 &  22&73\%\\ 
9 & 1 332 &  118 &  8&86\%\\ 
10 &  4 774 &  392 &  8&21\%\\ 
11 &  16 029 &  546 &  3&41\%\\ 
12 &  60 562 &  1 722 &  2&84\%\\ 
13 &  225 117 &  2 701 &  1&20\%\\ 
14 &  898 619 &  7 953 &  0&89\%\\ 
15 &  3 598 323 &  13 966 &  0&39\%\\ 
16 &  15 128 797 &  40 035 &  0&26\%\\ 
17 &  64 261 497 &  75 341&  0&12\%\\ 
18 &  283 239 174 &  210 763 &  0&07\%\\ 
 19 &  1 264 577 606 &  420 422 &  0&03\%\\ 
 20 &  5 817 868 002 &  1 162 192 &  0&02\%\\
 \bottomrule
\end{tabular} 
\caption{Comparison of the number of 3-edge-colourable pregraphs on $n$ vertices and the number of pregraphs with a \CQ on $n$ vertices.}\label{tab:ddgraphratio}
\end{center}
\end{table}

%%%%%%%%%%%%%%%%%%%%%%%%%%%%%%%%%%%%%%%%%%%%%%%%%

\section{$C^q_4$-marked pregraphs}\label{sec:cq4marked}

\begin{definition}
A \emph{\CQmarked pregraph} is a cubic pregraph in which all the edges of a given \CQ are coloured with colour 0 and all other edges with colour 1.\\
A \emph{\CQmarkable pregraph} is a cubic pregraph which has a \CQp.
\end{definition}

Note that the colouring in a \CQmarked pregraph clearly is not a proper edge-colouring as each vertex is incident to two edges of colour 0.
It is also clear that the underlying uncoloured graph of a \CQmarked pregraph is a pregraph which has a \CQp.
The following theorem shows that for a given $n$ for almost all \CQmarkable pregraphs on $n$ vertices there is a unique \CQ up to isomorphism.

First we repeat the definition of a few families of graphs that were already defined in \cite{pregraphs} and also define some new families.

Given a pregraph $P$, a {\em ladder} in $P$ is a maximal subgraph of $P$ that is isomorphic to the graph cartesian product of $K_2$ and the path $P_n$ with $n$ vertices for some $n\geq 2$.

A {\em prism} is a graph that is isomorphic to the graph cartesian product of $K_2$ and $C_n$ for some $n>2$.

A {\em M\"obius ladder} is a graph that is isomorphic to the graph on vertices $0,\dots,2n-1$ 
with vertex $i$ adjacent to $(i-1)\mbox{ mod }2n$, ~$(i+1)\mbox{ mod }2n$ ~and ~$(i+n)\mbox{ mod }2n$ where $n>1$.

A {\em crown} is a cycle with each vertex additionally incident to one semi-edge.

A \emph{barbed path} is a path where the end points are additionally incident to two semi-edges, and all other vertices are additionally incident to one semi-edge. See Figure~\ref{fig:path_multiple_CQ} for an example.

\begin{figure}
\begin{center}
\input{path_multiple_CQ.tex}
\end{center}

\caption{A barbed path has two different \CQs when its order is even and two isomorphic \CQs otherwise.}\label{fig:path_multiple_CQ}
\end{figure}
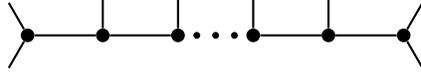

A \emph{double-closed ladder} is a graph isomorphic to the graph cartesian product of $K_2$ and the path $P_n$ with $n$ vertices for some $n\geq 2$, with the vertices corresponding to the same end point of the path additionally connected by an extra edge. See Figure~\ref{fig:ladders_multiple_CQ} for an example.

A \emph{double-open ladder} is a graph isomorphic to the graph cartesian product of $K_2$ and the path $P_n$ with $n$ vertices for some $n\geq 2$, with each vertex corresponding to an end point of the path additionally incident to a semi-edge. See Figure~\ref{fig:ladders_multiple_CQ} for an example.

An \emph{open-closed ladder} is a graph isomorphic to the graph cartesian product of $K_2$ and the path $P_n$ with $n$ vertices for some $n\geq 2$, with the vertices corresponding to one end point of the path additionally connected by an extra edge, and the vertices corresponding to the other end point additionally incident to a semi-edge. See Figure~\ref{fig:ladders_multiple_CQ} for an example.

\begin{figure}
\begin{center}
\input{ladder-2digons.tex}

\input{ladder-2semiedges.tex}

\input{ladder-digon-semiedge.tex}
\end{center}

\caption{The three families of graphs containing ladders that have multiple non-isomorphic \CQsp. From top to bottom: a double-closed ladder, a double-open ladder and an open-closed ladder.}\label{fig:ladders_multiple_CQ}
\end{figure}
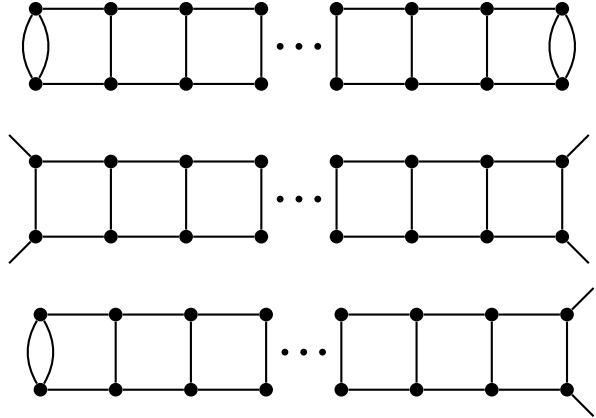

\begin{lemma}\label{lem:unique_partition}
A cubic pregraph $P$ on at least 4 vertices has a unique partition into ladders, subgraphs induced by digons not contained in a ladder and the components of the subgraph induced by the complement of these.
\end{lemma}
\begin{proof}
Since $P$ is cubic, the intersection of two ladders in $P$, respectively of two digons in $P$, is empty. The definition of the partition does not allow a vertex to be in two different types of parts of the partition. This proves the lemma.
\end{proof}

\begin{theorem}\label{thm:number_CQmarked}
For each integer $n>0$ the number of \CQmarkable pregraphs on $n$ vertices that have multiple pairwise non-isomorphic \CQsp, depends only on $n\bmod4$:
\begin{itemize}
\item $n\bmod 2 \equiv1$: 0 graphs
\item $n\bmod 4 \equiv0$: 4 graphs (the barbed path on $n$ vertices, the double-closed ladder on $n$ vertices, the double-open ladder on $n$ vertices and the open-closed ladder on $n$ vertices)
\item $n\bmod 4 \equiv2$: 2 graphs (the barbed path on $n$ vertices and the open-closed ladder on $n$ vertices)
\end{itemize}
Each \CQmarkable pregraph has at most two non-isomorphic \CQsp.
\end{theorem}
\begin{proof}
It is easily verified that the \CQ is unique for all \CQmarkable pregraphs on 1 and 3 vertices and that there are exactly 2 \CQmarkable pregraphs on 2 vertices that have 2 non-isomorphic \CQsp.

A crown on $n\geq 4$ vertices is \CQmarkable if $n$ is even. A \CQmarkable crown has two isomorphic \CQsp, corresponding to the two 1-factors of the cycle.

A M\"obius ladder on $n\geq 4$ vertices is \CQmarkable if $n$ is divisible by four. A M\"obius ladder on four vertices has three isomorphic \CQsp. A \CQmarkable M\"obius ladder on more than four vertices has two isomorphic \CQsp.

A prism on $n\geq 4$ vertices is \CQmarkable if $n$ is divisible by four. A prism on four vertices does not exist. A prism on eight vertices is a cube. It is easily seen that a cube has three isomorphic \CQsp: each \CQ corresponds to the edges of two opposite `faces' of the cube when viewed as a solid. A \CQmarkable prism on $n>8$ vertices has two isomorphic \CQsp.

Let $P$ be a pregraph on $n\geq4$ vertices which has a \CQ and is not a crown, M\"obius ladder or prism.
We will show that, except in a few cases, there is only one \CQ in $P$.
Owing to Lemma~\ref{lem:unique_partition}, $P$ has a unique partition into ladders, subgraphs induced by digons not contained in ladders and the components of the subgraph induced by the complement of these.

If $P$ contains a digon that is not contained in a ladder, then at least one of the two vertices $x$ and $y$ of the digon is incident to an edge $e$ that is not a semi-edge and that is not contained in the digon.
Since $e$ is not contained in a $C_4$ or a digon, and at least one of the vertices of $e$ is not incident to a semi-edge, $e$ can not be part of a quotient of $C_4$, and so the original digon is part of the \CQp.

If $P$ contains a ladder, then there are two possible situations.
The first case is that there is a boundary vertex $x$ of the ladder (a vertex not contained in an edge that is the intersection of two 4-cycles) that is incident to an edge $e$ that is not a semi-edge and is not contained in a $C_4$.
Using a similar argumentation as with the digon, it is easy to see that $e$ cannot be contained in a quotient of $C_4$.
This means that the $C_4$ which contains $x$ is part of any \CQ of $P$.
This also fixes any \CQ in this ladder.
The second case is that the ladder contains all the vertices of $P$, but $P$ is not a prism or a M\"obius ladder.
In this case we can look at the boundary vertices to determine the possible \CQsp.
Assume first that there are two boundary vertices which are connected by an edge $e$ which is not contained in a $C_4$ or a digon.
Again $e$ cannot be contained in a quotient of $C_4$ and this implies that also in this case the \CQ is unique.
There are still 3 graphs containing ladders which we have not discussed.
These graphs are the double-closed ladder, the double-open ladder and the open-closed ladder.

Let us first consider the double-closed ladder.
The edges at one side contained in a digon are in two ways contained in a quotient of $C_4$: either the digon itself is the quotient, or one of the edges of the digon together with the rest of the $C_4$ in which it is contained is the quotient.
In both cases the rest of the \CQ is unique for the whole graph.
In case $n\bmod 4 \equiv0$ this means that there are two non-isomorphic \CQmarked pregraphs, and in case $n\bmod 4 \equiv2$ there are two isomorphic \CQmarked pregraphs.

Next we look at the double-open ladder.
Again there are two ways the edges at one side can be contained in a \CQ and also in this case this means that there are two non-isomorphic \CQmarked pregraphs when $n\bmod 4 \equiv0$ and two isomorphic \CQmarked pregraphs when $n\bmod 4 \equiv2$.

Finally we have the open-closed ladder.
Again the edges at one side contained in a digon are in two ways contained in a quotient of $C_4$, but here these two \CQs are non-isomorphic  for all $n$.

The \CQ is fixed in ladders and in digons that are not part of a ladder.
The last step is to fix the \CQ in the rest of $P$.
By definition a vertex in this remainder is not contained in a digon or a $C_4$.
So for any \CQ of $P$, all vertices in the remainder will be in components isomorphic to $q_3$ or $q_4$.
This implies that each vertex in the remainder is incident to at least one semi-edge.
Since $P$ is not a crown, we find that this remainder consists of paths where each vertex is additionally also incident to at least one semi-edge.
First assume that an end vertex $x$ of such a path is incident to exactly one semi-edge.
As we saw earlier, this implies that in $P$ the vertex $x$ is incident to an edge $e$ that has a non-empty intersection with a digon or a ladder, and thus $e$ cannot be contained in a quotient of $C_4$.
This again fixes the \CQ for the whole path and we find that there is a unique \CQ for $P$ in this case.

So assume now that both end vertices are incident to exactly two semi-edges.
This is only possible if $P$ is a barbed path. If we look at one of the end vertices of $P$ in this case, we see that this vertex is contained in a quotient of $C_4$ in two ways: either the two semi-edges are the quotient, or one of the two semi-edges, the third edge and the semi-edge at the neighbouring vertex are the quotient.
Either choice fixes the \CQ for the whole graph.
In case $n$ is even, this means that there are two non-isomorphic \CQs in this pregraph, and in case $n$ is odd, both \CQs are isomorphic.
\end{proof}

Due to the previous theorem we can easily modify a generation algorithm for \CQmarked pregraphs to also generate pregraphs which have a \CQp. We just need to output the underlying graphs and make sure that we correctly handle the small number of graphs which lead to isomorphic unmarked pregraphs.

So let us look at how we can generate \CQmarked pregraphs.
We start by refining the unique partition into subgraphs induced by ladders, subgraphs induced by digons not contained in ladders and the components of the subgraph induced by the complement of these. For \CQmarked pregraphs, we can refine this partition such that each part contains only one type of quotients of $C_4$ (see Figure~\ref{fig:c4-quotients}).

\begin{definition}
Given a \CQmarked pregraph $P$, a \emph{block partition} of $P$ is a partition of $P$ into subgraphs of the following types:
\begin{enumerate}
\item maximal ladders containing only marked quotients of type $q_1$;
\item maximal subgraphs induced by marked quotients of type $q_2$;
\item maximal subgraphs induced by marked quotients of type $q_3$;
\item marked quotients of type $q_4$.
\end{enumerate}
Such a partition is unique for each \CQmarked pregraph $P$.

The different subgraphs in a block partition are called \emph{blocks}.
\end{definition}

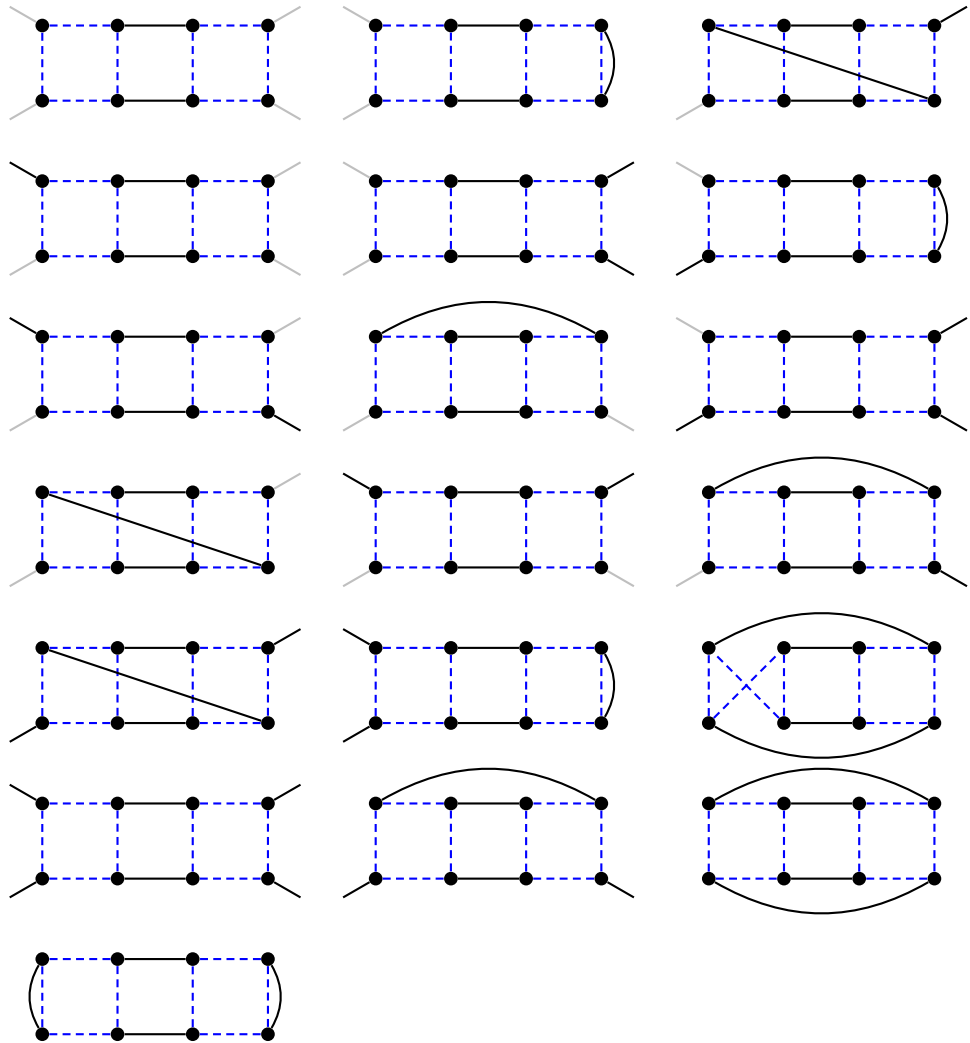
\begin{figure}
\begin{center}
\begin{minipage}[t]{0.3\textwidth}
\input{H.tex}

\input{LH.tex}

\input{DLH.tex}

\input{DC.tex}
\end{minipage}\hfill
\begin{minipage}[t]{0.3\textwidth}
\input{DHB.tex}

\input{OHB.tex}

\input{DLB.tex}

\input{OLB.tex}
\end{minipage}\hfill
\begin{minipage}[t]{0.3\textwidth}
\input{LDC.tex}

\input{LDHB.tex}

\input{LOHB.tex}

\input{LDLB.tex}
\end{minipage}

\begin{minipage}[t]{0.3\textwidth}
\input{DLDC.tex}

\input{CLH.tex}

\input{DDHB.tex}
\end{minipage}\hfill
\begin{minipage}[t]{0.3\textwidth}
\input{DLDHB.tex}

\input{DLDLB.tex}
\end{minipage}\hfill
\begin{minipage}[t]{0.3\textwidth}
\input{ML.tex}

\input{P.tex}
\end{minipage}
\end{center}
\caption{Representatives of all possible maximal ladders containing only marked quotients of type $q_1$. The dashed lines are the marked quotients of type $q_1$ and the grey lines are the edges that go to different blocks.}\label{fig:blocks_q1}
\end{figure}

\begin{figure}
\begin{center}

\begin{minipage}[t]{0.4\textwidth}
\begin{center}
\input{PC.tex}

\input{DLPC.tex}
\end{center}
\end{minipage}\hspace{20pt}
\begin{minipage}[t]{0.4\textwidth}
\begin{center}
\input{LPC.tex}

\def\blockparameter{5}
\input{PN.tex}
\end{center}
\end{minipage}
\end{center}
\caption{Representatives of all possible maximal subgraphs induced by marked quotients of type $q_2$. The dashed lines are the marked quotients of type $q_2$ and the grey lines are the edges that go to different blocks.}\label{fig:blocks_q2}
\end{figure}
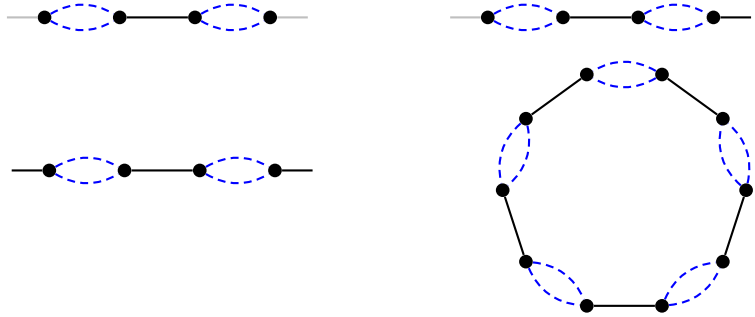

\begin{figure}
\begin{center}

\begin{minipage}[t]{0.4\textwidth}
\begin{center}
\input{BW.tex}

\input{DLBW.tex}
\end{center}
\end{minipage}\hspace{20pt}
\begin{minipage}[t]{0.4\textwidth}
\begin{center}
\input{LBW.tex}

\def\blockparameter{5}
\input{BWN.tex}
\end{center}
\end{minipage}
\end{center}
\caption{Representatives of all possible maximal subgraphs induced by marked quotients of type $q_3$. The dashed lines are the marked quotients of type $q_3$ and the grey lines are the edges that go to different blocks.}\label{fig:blocks_q3}
\end{figure}
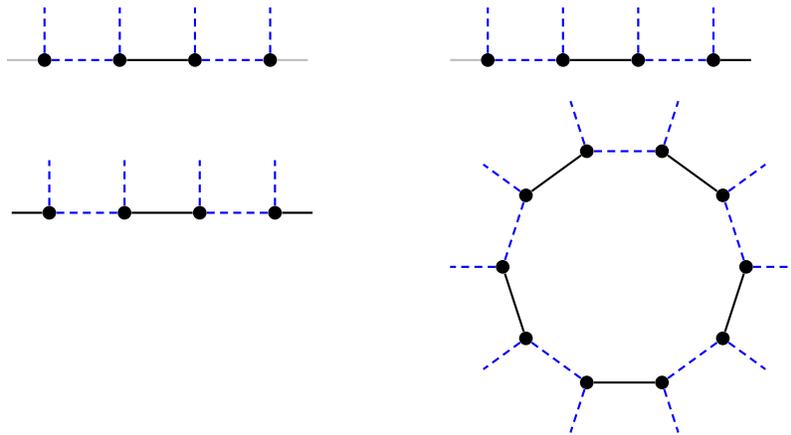

\begin{figure}
\begin{center}
\input{Q4.tex}
\hspace{80pt}
\input{T.tex}
\end{center}
\caption{All possible subgraphs induced by marked quotients of type $q_4$. The dashed lines are the marked quotients of type $q_4$ and the grey lines are the edges that go to different blocks.}\label{fig:blocks_q4}
\end{figure}
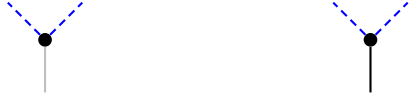

Figure~\ref{fig:blocks_q1} shows representatives of all possible maximal subgraphs induced by marked quotients of type $q_1$. These blocks can be determined by starting from a ladder with an order that is divisible by four and adding edges between the vertices of degree 2 or adding a semi-edge to these vertices. This was done by a straightforward ad-hoc script, since no specialised techniques are needed for this small set of blocks.
Figure~\ref{fig:blocks_q2}, respectively Figure~\ref{fig:blocks_q3}, shows representatives of all possible maximal subgraphs induced by marked quotients of type $q_2$, respectively of type $q_3$.
Figure~\ref{fig:blocks_q4} shows all possible subgraphs induced by marked quotients of type $q_4$.

We use block partitions to generate the \CQmarked pregraphs.
This generation process happens in several phases.
To generate all the \CQmarked pregraphs with $n$ vertices, we start by generating all lists of blocks such that the sum of the orders of the blocks is equal to $n$.
This is done by a simple orderly algorithm without many optimisations since the time spent in the generation of these lists is negligible compared to the rest of the generation process.
Since each \CQmarked pregraph corresponds to exactly one such list of blocks, different lists will result in different \CQmarked pregraphs.

We can perform a few tests to discard lists that are not realisable as a block partition of a \CQmarked pregraph.
We can view a \CQmarked pregraph with a block partition as a multigraph: the vertices of the multigraph are the blocks of the block partition, the edges of the multigraph are the edges between the blocks.
This means that when we have a list of blocks we can check whether the degree sequence that corresponds to that list is realisable as a multigraph without loops.
We use the characterisation by Owens and Trent \cite{OT:67} to check whether a degree sequence is realisable as a loopless multigraph.

Not all lists for which the corresponding degree sequences are realisable will occur as block partition of a \CQmarkable pregraph.
Another restriction we need to take into account when connecting the blocks is that not every connection is valid.
A condition for the unicity of the block partition was that the blocks were maximal.
This means that connecting a $q_2$ block to a $q_2$ block or a $q_3$ block to a $q_3$ block is not allowed.
Neither is it allowed to connect two $q_1$ blocks such that two of the connecting edges are part of a $C_4$.
So we can add the following test to discard some more blocks: if more than half of the connections are connections at $q_2$ blocks or more than half of the connections are connections at $q_3$ blocks, then this list will not be realisable.

A list $L$ of blocks is \emph{acceptable} if and only if
\begin{itemize}
\item the degree sequence corresponding to $L$ is multigraphic,
\item half or less than half of the missing connections lie in $q_2$ blocks, and
\item half or less than half of the missing connections lie in $q_3$ blocks.
\end{itemize}

Once we have a list we need to try and add the connections in all possible ways.

\begin{definition}
A \emph{\pCQmarkedpg} is a not necessarily connected pregraph $P$ in which all the edges of a given \CQ are coloured with colour 0 and all other edges with colour 1 and where all vertices have either degree 2 or degree 3.

The vertices with degree 2 are called the \emph{deficient vertices} of $P$.
\end{definition}

\begin{definition}
Denote by \emph{$B(P)$} the block list corresponding to the unique block partition of $P$. 

Given a block list $L$, denote by \emph{$\mathcal{B}_L$} the set of \pCQmarkedpgs with a block partition isomorphic to $L$.
\end{definition}

A block list corresponds in a trivial way to a \pCQmarkedpgp.
We construct a new \pCQmarkedpg by connecting two deficient vertices.
The new \pCQmarkedpg has two deficient vertices less than the original graph.
Once we have a \pCQmarkedpg with no deficient vertices we have found a \CQmarked pregraph.
To generate the \CQmarked pregraphs we use the principle of closed structures \cite{Br09,Br:13}.

\begin{definition}
A \emph{marked subgraph} of a \pCQmarkedpg $P$ is a subgraph $P_s$ of $P$ such that $P_s$ is a \pCQmarkedpg and the colours of the edges in $P_s$ is the same as the colours of the edges in $P$. 

A \pCQmarkedpg $P'$ is an \emph{extension} of a \pCQmarkedpg $P$ if $P$ is a marked subgraph of $P'$, and $P'$ and $P$ have the same number of vertices.

A \pCQmarkedpg $P$ is \emph{closed} if for any two extensions $P_1$ and $P_2$ of $P$ we have that any isomorphism between $P_1$ and $P_2$ induces an automorphism of $P$.
\end{definition}

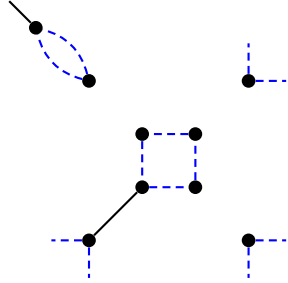
\begin{figure}
\begin{center}
\input{not_closed_partial_pregraph.tex}
\end{center}
\caption{A \pCQmarkedpg that is not closed.}\label{fig:not_closed_partialpg}
\end{figure}

The \pCQmarkedpg $P$ in Figure~\ref{fig:not_closed_partialpg} is not closed since the two extensions in Figure~\ref{fig:not_closed_partialpg_exts} are isomorphic, but the isomorphism maps the vertex in the bottom right to the vertex in the bottom left and vice versa.
Clearly this does not induce an automorphism of $P$ since these vertices have different degrees in $P$.

\begin{figure}
\begin{center}
{\hfill
\input{not_closed_partial_pregraph_ext1.tex}\hfill
\input{not_closed_partial_pregraph_ext2.tex}
\hfill { }
}
\end{center}
\caption{Two extensions of the \pCQmarkedpg in Figure~\ref{fig:not_closed_partialpg} which are isomorphic but for which the isomorphism does not induce an automorphism of the original \pCQmarkedpgp.}\label{fig:not_closed_partialpg_exts}
\end{figure}
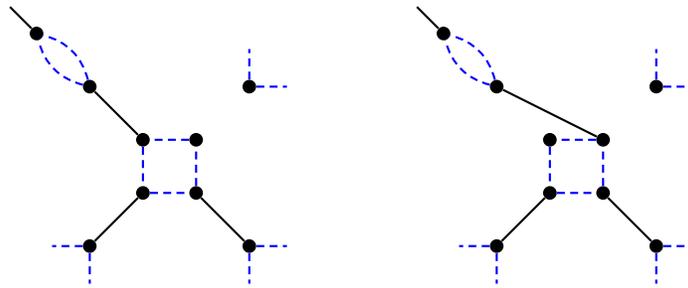

The \pCQmarkedpg corresponding to a block list $L$ is a closed \pCQmarkedpg since any connection that would create a subgraph that is a block leads to a \pCQmarkedpg that is not in $\mathcal{B}_L$.

The advantage of this technique with closed structures is that if the closed \pCQmarkedpg $P$ has a trivial symmetry group, no two different extensions of $P$ will be isomorphic, and so no isomorphism rejection is needed once a closed graph with trivial symmetry appears in the generation process.
Clearly we want to reach such a closed \pCQmarkedpg as soon as possible during the generation.
The following lemma shows a way how a new closed \pCQmarkedpg can be obtained when starting with a closed \pCQmarkedpgp.

\begin{lemma}\label{lem:closedpg}
Let $P$ be a closed \pCQmarkedpgp. Let $O$ be an orbit of deficient vertices under the automorphism group of $P$. Let $P'$ be an extension of $P$ so that no vertex from $O$ is deficient and no edges in $P'\setminus P$ have an empty intersection with $O$. Then $P'$ is also a closed \pCQmarkedpgp.
\end{lemma}
\begin{proof}
Assume we have two extensions $P_1'$ and $P_2'$ of $P'$.
We need to prove that if there is an isomorphism $\sigma$ between $P_1'$ and $P_2'$, this isomorphism induces an automorphism of $P'$.
Since $P$ is closed, we have that $\sigma$ induces an automorphism of $P$.

Assume that there is an isomorphism $\sigma: P_1' \rightarrow P_2'$ that does not induce an automorphism of $P'$.
So there exist vertices $x,y$ in $P'$, such that $x$ and $y$ are adjacent in $P'$, and $\sigma(x)$ and $\sigma(y)$ are not adjacent in $P'$.
As $P$ is a subgraph of $P'$, it cannot be that $\sigma(x)$ and $\sigma(y)$ are adjacent in $P$, so we have that $x$ and $y$ are non-adjacent in $P$. This means at least one of the two vertices $x$ and $y$ belongs to $O$, so assume that $x\in O$.
As $\sigma$ induces an automorphism on $P$, this means that $\sigma(x)\in O$ and so in $P'$ $\sigma(x)$ is adjacent to another vertex, say $z$.
Since $x$ is adjacent to $y$ in $P_1'$, $\sigma(x)$ is adjacent to $\sigma(y)$ in $P_2'$, but this contradicts that $\sigma(x)$ is not a deficient vertex in $P'$.

So we find that for any two extensions $P_1'$ and $P_2'$, there is no isomorphism between $P_1'$ and $P_2'$ that does not induce an automorphism of $P'$ and thus $P'$ is a closed \pCQmarkedpgp.
\end{proof}

We want to go from closed \pCQmarkedpgs to closed \pCQmarkedpgs as `fast' as possible, that is why we will each time select the smallest orbit of deficient vertices and add connections to that orbit.
Once we find a \pCQmarkedpg with a trivial symmetry, we can stop any isomorphism rejections and just add the remaining connections in all possible ways.
We can even do this a bit sooner: it is sufficient that the symmetry group acts trivially on the deficient vertices.

\begin{theorem}
Let $P$ be a closed \pCQmarkedpgp. If the automorphism group of $P$ acts trivially on the deficient vertices of $P$, then all extensions of $P$ are pairwise non-isomorphic.
\end{theorem}
\begin{proof}
Let $P_1$ and $P_2$ be two different extensions of $P$. Assume that there is an isomorphism $\sigma$ between $P_1$ and $P_2$. Since $P$ is closed, the isomorphism $\sigma$ induces an automorphism of $P$. Let $e$ be an edge in $P_1\setminus P$. Both vertices incident to $e$ are deficient vertices in $P$, and so they are fixed by $\sigma$. This implies that the edge $e$ is fixed by $\sigma$, so we find that all edges in $P_1\setminus P$ are fixed by $\sigma$, which contradicts that $P_1$ and $P_2$ are different extensions.
\end{proof}

In summary, we use the following algorithm to generate the \CQmarked pregraphs with $n$ vertices:

\begin{enumerate}
\item Generate all acceptable lists of blocks such that the sum of the orders of the blocks is $n$.
\item For each list construct the corresponding \pCQmarkedpg $P$ and recursively repeat the following steps :
\begin{enumerate}
\item If $P$ has no deficient vertices: output $P$ and return.
\item Compute the automorphism group of $P$ and compute the orbits of deficient vertices of $P$.
\item If $Aut(P)$ acts trivial on the set of deficient vertices of $P$, then complete $P$ by adding the remaining connections in all possible valid ways and output any complete \CQmarked pregraph obtained this way.
\item \label{alg:pg_orbit_selection} Otherwise choose the smallest orbit $O$ and connect all vertices in $O$ to deficient vertices in all valid ways that give non-isomorphic \pCQmarkedpg and repeat these steps for this new \pCQmarkedpgp.
\end{enumerate}
\end{enumerate}

The fact that this algorithm moves from closed \pCQmarkedpg to closed \pCQmarkedpg is not sufficient to guarantee that no pairwise isomorphic structures are output. Although all extensions of a closed \pCQmarkedpg are pairwise non-isomorphic, it might still be possible that they are isomorphic to extensions of another \pCQmarkedpgp. That this is not the case still needs to be proven.

\begin{definition}
A \pCQmarkedpg $P$ is \emph{strongly closed} in a set $S$ of \pCQmarkedpgsp, if all \pCQmarkedpgs in $S$ that contain a subgraph isomorphic to $P$ are extensions of $P$.
\end{definition}

\begin{lemma}\label{lem:stronglyclosedpg}
Let $L$ be the block list under consideration in Step 2 in the algorithm above.
The \pCQmarkedpgs to which the recursive step is applied, are strongly closed in $\mathcal{B}_L$.
\end{lemma}
\begin{proof}
Due to Lemma~\ref{lem:closedpg}, the graphs to which the recursive step is applied are closed. It is also clear that the initial \pCQmarkedpg that corresponds to the block list $L$ without connections is strongly closed in the set $\mathcal{B}_L$.

What remains to be proven is that if a \pCQmarkedpg $P$ is strongly closed in $\mathcal{B}_{B(P)}$, $O$ is an orbit of deficient vertices of $P$ under the automorphism group of $P$ and $P'$ is an extension of $P$ such that all vertices in $O$ are no longer deficient and no edges were added that have an empty intersection with $O$, then $P'$ is also strongly closed in $\mathcal{B}_{B(P)}$.

Given a \pCQmarkedpg $P''$ such that $B(P'')=B(P)$ and that $P''$ contains a subgraph $P'_s$ that is isomorphic to $P'$. As $P'$ is an extension of $P$, $P'_s$ also contains a subgraph that is isomorphic to $P$, and thus $P'_s$ is an extension of $P$. Since $P$ is closed, we have that the isomorphism between $P'_s$ and $P'$ induces an automorphism of $P$. Since $O$ is an orbit under the automorphism group of $P$, $O$ is mapped onto $O$. This means that both $P'$ and $P'_s$ are extensions of $P$ and for both \pCQmarkedpgs the same orbit of deficient vertices was chosen in step~\ref{alg:pg_orbit_selection} of the algorithm above. In step~\ref{alg:pg_orbit_selection} only pairwise non-isomorphic \pCQmarkedpgs are generated, so we find that $P'=P'_s$. So we have that $P''$ is an extension of $P'$, which proves that $P'$ is strongly closed in $\mathcal{B}_{B(P)}$.
\end{proof}

\begin{theorem}
The algorithm above outputs exactly one representative of every isomorphism class of \CQmarkedpg with $n$ vertices.
\end{theorem}
\begin{proof}
This theorem follows from Lemma~\ref{lem:closedpg} and Lemma~\ref{lem:stronglyclosedpg}, together with the fact that each \CQmarkedpg has a unique block partition.
\end{proof}

When we construct the \pCQmarkedpg $P$ corresponding to a list we also construct the automorphism group of $P$, i.e., we construct a set of generators for the automorphism group based upon the automorphisms of the blocks and the isomorphism of similar blocks. For further computations of the automorphism group we use the program {\em nauty} \cite{McK81__2}.

For step~\ref{alg:pg_orbit_selection} we use McKay's canonical construction path method \cite{McK96}. Given a \pCQmarkedpg $P$ and an orbit $O$ of deficient vertices, we first calculate the orbits of unordered pairs of deficient vertices $\{x,y\}$ such that $\{x,y\}\cap O$ is not empty. For each orbit of unordered pairs we choose one pair in that orbit and connect these vertices if this is a valid connection. There are two reasons why a connection could be invalid: it might create a new block, or it might create a subgraph which does not contain all the vertices and does not contain any deficient vertices. In case this is a valid connection, we still need to verify that it is the canonical operation to obtain the resulting \pCQmarkedpg $P'$. This is done by labelling each vertex $v$ with a 2-tuple $(x_1,x_2)$. In this tuple $x_1$ is the label of $v$ in a canonical labelling of $P$ and $x_2$ is the label of $v$ in a canonical labelling of $P'$. This operation is accepted if and only if the new connection is in the orbit of connections in $P'$ which have a non-empty intersection with $O$ and for which the vertices have the lexicographically smallest vertex labels. It is often not needed to construct a canonical labelling of $P'$, since the operation can already be discarded as being not canonical based on the values of $x_1$ for the vertices.

%%%%%%%%%%%%%%%%%%%%%%%%%%%%%%%%%%%%%%%%%%%%%%%%%

\section{Delaney-Dress graphs}

Given a Delaney-Dress graph $G$ we can easily construct a \CQmarked pregraph $P$ from $G$ by marking the edges with colour 0 and colour 2 and then removing all original colours.
When we want to generate Delaney-Dress graphs from \CQmarked pregraphs, then we want to go in the other direction, i.e., we need to assign colours 0 and 2 to the marked quotients of $C_4$ in the \CQmarked pregraph.
Clearly the construction above leads to a unique \CQmarked pregraph corresponding to a Delaney-Dress graph, and so different \CQmarked pregraphs will lead to different Delaney-Dress graphs.

We need to check that different colour assignments do not lead to isomorphic Delaney-Dress graphs.
In the cases where this does happen, we only accept one of these isomorphic copies.

A first observation we can make is that if we swap the colours in a quotient of type $q_2$ or in a quotient of type $q_4$, we always get an isomorphic Delaney-Dress graph.
We indeed always have the isomorphism that fixes all the vertices and all the edges that are not in that quotient and that interchanges the two edges, resp. semi-edges, in that quotient.
This means that we can just choose an arbitrary colouring for these quotients and can focus the isomorphism rejection on the quotients of type $q_1$ and the quotients of type $q_3$.

Given a partially coloured Delaney-Dress graph $D$ such that the uncoloured subgraphs are quotients of type $q_1$ and of type $q_3$, the set of uncoloured quotients is denoted by $U$. Note that if $n$ is the number of vertices in $D$, then $U$ contains at most $\frac{n}{2}$ elements and in most cases it will be much less than that. We can define a bijection between the set of valid colour assignments for $D$ and the set of binary vectors with length $|U|$. We start by choosing a matching (i.e., two non-adjacent edges) in each quotient of type $q_1$. For a quotient $u$ of type $q_1$ we denote this matching by $m(u)$. We also label the quotients in $U$ with the numbers 1 to $|U|$.

A colouring $c$ is mapped to a binary vector $v_c$ as follows. The $i$th coordinate of $v_c$ corresponds to the uncoloured quotient $u\in U$ that has label $i$. If $u$ is of type $q_1$, then the $i$th coordinate of $v_c$ is equal to 0 if the edges in $m(u)$ receive colour 0, and is equal to 1 if these edges receive colour 2. If $u$ is of type $q_3$, then the $i$th coordinate of $v_c$ is equal to 0 if the semi-edges in $u$ receive colour 0, and is equal to 0 otherwise.

Given an automorphism $\sigma$ of $D$ and a binary vector $v_c$ corresponding to a colouring $c$, we can easily construct the binary vector $v_c'$ that corresponds to the colouring $c'$ of $D$ when we would apply $\sigma$ to the coloured Delaney-Dress graph.
The automorphism $\sigma$ will map a quotient $u\in U$ to another quotient $u' \in U$, and clearly $u$ and $u'$ will be of the same type.
If for each vertex the factor in which it is contained is known, it is sufficient to know the image of one vertex of $u$ to determine $u'$.
In case $u$ is of type $q_3$, then the coordinate in $v_c'$ corresponding to $u'$ will have the same value as the coordinate in $v_c$ corresponding to $u$.
In case $u$ is of type $q_1$, then we need to check whether $m(u)$ is mapped to $m(u')$ by $\sigma$.
(For this it is sufficient to know the image of one edge of $m(u)$.)
If this is the case, then the coordinate in $v_c'$ corresponding to $u'$ will have the same value as the coordinate in $v_c$ corresponding to $u$.
Otherwise they will have different values.

This means that we can perform the orbit calculations on the set of binary vectors. We use the \emph{union-find} algorithm on the set of all binary vectors with length $|U|$ to find which coloured Delaney-Dress graphs are isomorphic.

%%%%%%%%%%%%%%%%%%%%%%%%%%%%%%%%%%%%%%%%%%%%%%%%%

\section{Testing}

Small errors are always easily made in both mathematical proofs and computer programs.
In computer programs however, they are often more concealed and less subject to scrutinous checking.
It is therefore important to perform tests of the computer programs.
Preferably using indepently written programs based on different algorithms.

The numbers of \CQmarkable pregraphs up to 20 vertices have been compared to the numbers obtained in \cite{pregraphs}.
Since the techniques used in both cases are very different this offers an independent test for the implementation. 
The program used in \cite{pregraphs} had already itself been compared to manual results.

The numbers of Delaney-Dress graphs up to $n=10$ vertices have been compared to the results obtained in \cite{HOPF:13}.

%%%%%%%%%%%%%%%%%%%%%%%%%%%%%%%%%%%%%%%%%%%%%%%%%

\section{Results}

The algorithms described in this article have been implemented in C as the program \texttt{ddgraphs}. It is available at \cite{ddgraphssite}.

Table~\ref{tab:overview_ddgraphs_CQmarkable} gives an overview of the numbers of block lists, the numbers of \CQmarked pregraphs and the numbers of \CQmarkable pregraphs on up to $n=35$ vertices.

Table~\ref{tab:overview_ddgraphs} shows an overview of the numbers of Delaney-Dress graphs on up to $n=35$ vertices.

The numbers for graphs on more than 30 vertices were obtained by splitting the generation is several parts.
This was done by generating the block lists and distributing which block lists needed to be extended.
Since the generation of the block lists is negligable compared to the remaining generation process, this splitting in parts can be done very efficiently.

\begin{landscape}
\begin{longtable}{rrrrrrrr}
\caption[An overview of the number of block lists, the number of \CQmarked pregraphs and the number of \CQmarkable pregraphs with $n$ vertices.]{An overview of the number of block lists, the number of \CQmarked pregraphs and the number of \CQmarkable pregraphs with $n$ vertices. For each coloumn the time needed to generate those structures using the program \texttt{ddgraphs} is given. For the \CQmarkable pregraphs also the time needed by \texttt{pregraphs} is given \cite{pregraphs}. All timings were done on a 2.40 GHz Intel Xeon.}\label{tab:overview_ddgraphs_CQmarkable}\\
\toprule
\cth{$n$} & \cth{lists} & \cth{time} & \cth{\CQmarked} & \cth{time} & \cth{\CQmarkable} & \cth{time} & \cth{time} \\
 &  & \cth{\texttt{ddgraphs}} &  & \cth{\texttt{ddgraphs}} &  & \cth{\texttt{ddgraphs}} & \cth{\texttt{pregraphs}} \\
\midrule
\endfirsthead
\caption[]{An overview of the number of block lists, the number of \CQmarked pregraphs and the number of \CQmarkable pregraphs with $n$ vertices. (Continued)}\\
\midrule
\cth{$n$} & \cth{lists} & \cth{time} & \cth{\CQmarked} & \cth{time} & \cth{\CQmarkable} & \cth{time} & \cth{time} \\
 &  & \cth{\texttt{ddgraphs}} &  & \cth{\texttt{ddgraphs}} &  & \cth{\texttt{ddgraphs}} & \cth{\texttt{pregraphs}} \\
\midrule
\endhead
\midrule  \multicolumn{8}{r}{{Continued on next page}} \\ 
\endfoot
\bottomrule
\endlastfoot
\input{table_overview_lists-marked-unmarked_count_time_approx.tex}

\end{longtable}
\end{landscape}

\begin{table}
\begin{center}
\input{table_ddgraphs_approx.tex}
\end{center}
\caption{An overview of the number of Delaney-Dress graphs and the time needed by \texttt{ddgraphs} to generate these graphs when run on a 2.40 GHz Intel Xeon.}\label{tab:overview_ddgraphs}
\end{table}

%%%%%%%%%%%%%%%%%%%%%%%%%%%%%%%%%%%%%%%%%%%%%%%%%

\bibliographystyle{plainyr}
%\bibliography{../biblio}

\end{document}

%% file: q-1.tex
\begin{tikzpicture}
\useasboundingbox (-0.5,-0.25) rectangle (1.5,1.25);
\node [vertex, scale=0.75] (a) at (0,0) {};
\node [vertex, scale=0.75] (b) at (1,0) {};
\node [vertex, scale=0.75] (c) at (0,1) {};
\node [vertex, scale=0.75] (d) at (1,1) {};
\draw [thick] (a) to (b);
\draw [thick] (a) to (c);
\draw [thick] (d) to (b);
\draw [thick] (d) to (c);
\end{tikzpicture}

%% file: q-2.tex
\begin{tikzpicture}
\useasboundingbox (-0.5,-0.75) rectangle (1.5,0.75);
\node [vertex, scale=0.75] (a) at (0,0) {};
\node [vertex, scale=0.75] (b) at (1,0) {};
\draw [thick, bend left] (a) to (b);
\draw [thick, bend right] (a) to (b);
\end{tikzpicture}

%% file: q-3.tex
\begin{tikzpicture}
\useasboundingbox (-0.5,-0.75) rectangle (1.5,0.75);
\node [vertex, scale=0.75] (a) at (0,0) {};
\node [vertex, scale=0.75] (b) at (1,0) {};
\draw [thick] (a) to (b);
\draw [thick] (a) to (0,0.5);
\draw [thick] (b) to (1,0.5);
\end{tikzpicture}

%% file: q-4.tex
\begin{tikzpicture}
\useasboundingbox (-1,-0.75) rectangle (1,0.75);
\node [vertex, scale=0.75] (a) at (0,0) {};
\draw [thick] (a) to (-0.3,0.5);
\draw [thick] (a) to (0.3,0.5);
\end{tikzpicture}

%% file: path_multiple_CQ.tex
\ifx \tikzscale \undefined
   \def \tikzscale {1}
\fi%
\begin{tikzpicture}[scale=\tikzscale,baseline=-2]

\pgfmathsetmacro\vertexscale{0.5*\tikzscale}
\pgfmathsetmacro\dotscale{0.25*\tikzscale}

\useasboundingbox (-0.5,-1) rectangle (5.5,1);

\node [vertex,scale=\dotscale] (dot1) at (2.25,0) {};
\node [vertex,scale=\dotscale] (dot2) at (2.5,0) {};
\node [vertex,scale=\dotscale] (dot3) at (2.75,0) {};

\node [fill,circle, scale=\vertexscale] (a) at (0,0) {};
\node [fill,circle, scale=\vertexscale] (b) at (1,0) {};
\node [fill,circle, scale=\vertexscale] (c) at (2,0) {};
\node [fill,circle, scale=\vertexscale] (d) at (3,0) {};
\node [fill,circle, scale=\vertexscale] (e) at (4,0) {};
\node [fill,circle, scale=\vertexscale] (f) at (5,0) {};
\draw [thick, black] (a) -- ++(120:0.5);
\draw [thick, black] (a) -- ++(240:0.5);
\draw [thick, black] (a) to (b);
\draw [thick, black] (b) to (c);
\draw [thick, black] (b)  -- ++(90:0.5);
\draw [thick, black] (c)  -- ++(90:0.5);
\draw [thick, black] (d) to (e);
\draw [thick, black] (d)  -- ++(90:0.5);
\draw [thick, black] (e)  -- ++(90:0.5);
\draw [thick, black] (e) to (f);
\draw [thick, black] (f) -- ++(60:0.5);
\draw [thick, black] (f) -- ++(300:0.5);

\end{tikzpicture}

%% file: ladder-2digons.tex
\ifx \tikzscale \undefined
   \def \tikzscale {1}
\fi%
\tikzset{%
  sigma1style/.prefix style={thick,red},%
  connectorstyle/.prefix style={thick,red,densely dotted},%
  connectormarkstyle/.prefix style={thick,red},%
  unset2factorstyle/.prefix style={very thick,blue,densely dashed}%
}

\begin{tikzpicture}[baseline=(current bounding box.center),scale=\tikzscale]
\useasboundingbox (-5,-1) rectangle (4,1);

\pgfmathsetmacro\vertexscale{0.5*\tikzscale}
\pgfmathsetmacro\dotscale{0.25*\tikzscale}

\node [vertex,scale=\dotscale] (dot1) at (-0.75,0) {};
\node [vertex,scale=\dotscale] (dot2) at (-0.5,0) {};
\node [vertex,scale=\dotscale] (dot3) at (-0.25,0) {};
\node [vertex,scale=\vertexscale] (a) at (0,0.5) {};
\node [vertex,scale=\vertexscale] (b) at (0,-0.5) {};
\node [vertex,scale=\vertexscale] (c) at (1,0.5) {};
\node [vertex,scale=\vertexscale] (d) at (1,-0.5) {};
\node [vertex,scale=\vertexscale] (e) at (2,0.5) {};
\node [vertex,scale=\vertexscale] (f) at (2,-0.5) {};
\node [vertex,scale=\vertexscale] (g) at (3,0.5) {};
\node [vertex,scale=\vertexscale] (h) at (3,-0.5) {};

\node [vertex,scale=\vertexscale] (aLeft) at (-1,0.5) {};
\node [vertex,scale=\vertexscale] (bLeft) at (-1,-0.5) {};
\node [vertex,scale=\vertexscale] (cLeft) at (-2,0.5) {};
\node [vertex,scale=\vertexscale] (dLeft) at (-2,-0.5) {};
\node [vertex,scale=\vertexscale] (eLeft) at (-3,0.5) {};
\node [vertex,scale=\vertexscale] (fLeft) at (-3,-0.5) {};
\node [vertex,scale=\vertexscale] (gLeft) at (-4,0.5) {};
\node [vertex,scale=\vertexscale] (hLeft) at (-4,-0.5) {};

\draw [thick] (a) to (b);
\draw [thick] (a) to (c);
\draw [thick] (b) to (d);
\draw [thick] (c) to (d);
\draw [thick] (e) to (f);
\draw [thick] (c) to (e);
\draw [thick] (d) to (f);
\draw [thick] (e) to (g);
\draw [thick] (f) to (h);
\draw [thick,bend left] (g) to (h);
\draw [thick,bend left] (h) to (g);

\draw [thick] (aLeft) to (bLeft);
\draw [thick] (aLeft) to (cLeft);
\draw [thick] (bLeft) to (dLeft);
\draw [thick] (cLeft) to (dLeft);
\draw [thick] (eLeft) to (fLeft);
\draw [thick] (cLeft) to (eLeft);
\draw [thick] (dLeft) to (fLeft);
\draw [thick] (eLeft) to (gLeft);
\draw [thick] (fLeft) to (hLeft);
\draw [thick,bend left] (gLeft) to (hLeft);
\draw [thick,bend left] (hLeft) to (gLeft);

\end{tikzpicture}

%% file: ladder-2semiedges.tex
\ifx \tikzscale \undefined
   \def \tikzscale {1}
\fi%
\tikzset{%
  sigma1style/.prefix style={thick,red},%
  connectorstyle/.prefix style={thick,red,densely dotted},%
  connectormarkstyle/.prefix style={thick,red},%
  unset2factorstyle/.prefix style={very thick,blue,densely dashed}%
}

\begin{tikzpicture}[baseline=(current bounding box.center),scale=\tikzscale]
\useasboundingbox (-5,-1) rectangle (4,1);

\pgfmathsetmacro\vertexscale{0.5*\tikzscale}
\pgfmathsetmacro\dotscale{0.25*\tikzscale}

\node [vertex,scale=\dotscale] (dot1) at (-0.75,0) {};
\node [vertex,scale=\dotscale] (dot2) at (-0.5,0) {};
\node [vertex,scale=\dotscale] (dot3) at (-0.25,0) {};
\node [vertex,scale=\vertexscale] (a) at (0,0.5) {};
\node [vertex,scale=\vertexscale] (b) at (0,-0.5) {};
\node [vertex,scale=\vertexscale] (c) at (1,0.5) {};
\node [vertex,scale=\vertexscale] (d) at (1,-0.5) {};
\node [vertex,scale=\vertexscale] (e) at (2,0.5) {};
\node [vertex,scale=\vertexscale] (f) at (2,-0.5) {};
\node [vertex,scale=\vertexscale] (g) at (3,0.5) {};
\node [vertex,scale=\vertexscale] (h) at (3,-0.5) {};

\node [vertex,scale=\vertexscale] (aLeft) at (-1,0.5) {};
\node [vertex,scale=\vertexscale] (bLeft) at (-1,-0.5) {};
\node [vertex,scale=\vertexscale] (cLeft) at (-2,0.5) {};
\node [vertex,scale=\vertexscale] (dLeft) at (-2,-0.5) {};
\node [vertex,scale=\vertexscale] (eLeft) at (-3,0.5) {};
\node [vertex,scale=\vertexscale] (fLeft) at (-3,-0.5) {};
\node [vertex,scale=\vertexscale] (gLeft) at (-4,0.5) {};
\node [vertex,scale=\vertexscale] (hLeft) at (-4,-0.5) {};

\draw [thick] (a) to (b);
\draw [thick] (a) to (c);
\draw [thick] (b) to (d);
\draw [thick] (c) to (d);
\draw [thick] (e) to (f);
\draw [thick] (c) to (e);
\draw [thick] (d) to (f);
\draw [thick] (e) to (g);
\draw [thick] (f) to (h);
\draw [thick] (g) to (h);
\draw [thick] (g) -- ++(45:0.5);
\draw [thick] (h) -- ++(-45:0.5);

\draw [thick] (aLeft) to (bLeft);
\draw [thick] (aLeft) to (cLeft);
\draw [thick] (bLeft) to (dLeft);
\draw [thick] (cLeft) to (dLeft);
\draw [thick] (eLeft) to (fLeft);
\draw [thick] (cLeft) to (eLeft);
\draw [thick] (dLeft) to (fLeft);
\draw [thick] (eLeft) to (gLeft);
\draw [thick] (fLeft) to (hLeft);
\draw [thick] (gLeft) to (hLeft);
\draw [thick] (gLeft) -- ++(135:0.5);
\draw [thick] (hLeft) -- ++(-135:0.5);

\end{tikzpicture}

%% file: ladder-digon-semiedge.tex
\ifx \tikzscale \undefined
   \def \tikzscale {1}
\fi%
\tikzset{%
  sigma1style/.prefix style={thick,red},%
  connectorstyle/.prefix style={thick,red,densely dotted},%
  connectormarkstyle/.prefix style={thick,red},%
  unset2factorstyle/.prefix style={very thick,blue,densely dashed}%
}

\begin{tikzpicture}[baseline=(current bounding box.center),scale=\tikzscale]
\useasboundingbox (-5,-1) rectangle (4,1);

\pgfmathsetmacro\vertexscale{0.5*\tikzscale}
\pgfmathsetmacro\dotscale{0.25*\tikzscale}

\node [vertex,scale=\dotscale] (dot1) at (-0.75,0) {};
\node [vertex,scale=\dotscale] (dot2) at (-0.5,0) {};
\node [vertex,scale=\dotscale] (dot3) at (-0.25,0) {};
\node [vertex,scale=\vertexscale] (a) at (0,0.5) {};
\node [vertex,scale=\vertexscale] (b) at (0,-0.5) {};
\node [vertex,scale=\vertexscale] (c) at (1,0.5) {};
\node [vertex,scale=\vertexscale] (d) at (1,-0.5) {};
\node [vertex,scale=\vertexscale] (e) at (2,0.5) {};
\node [vertex,scale=\vertexscale] (f) at (2,-0.5) {};
\node [vertex,scale=\vertexscale] (g) at (3,0.5) {};
\node [vertex,scale=\vertexscale] (h) at (3,-0.5) {};

\node [vertex,scale=\vertexscale] (aLeft) at (-1,0.5) {};
\node [vertex,scale=\vertexscale] (bLeft) at (-1,-0.5) {};
\node [vertex,scale=\vertexscale] (cLeft) at (-2,0.5) {};
\node [vertex,scale=\vertexscale] (dLeft) at (-2,-0.5) {};
\node [vertex,scale=\vertexscale] (eLeft) at (-3,0.5) {};
\node [vertex,scale=\vertexscale] (fLeft) at (-3,-0.5) {};
\node [vertex,scale=\vertexscale] (gLeft) at (-4,0.5) {};
\node [vertex,scale=\vertexscale] (hLeft) at (-4,-0.5) {};

\draw [thick] (a) to (b);
\draw [thick] (a) to (c);
\draw [thick] (b) to (d);
\draw [thick] (c) to (d);
\draw [thick] (e) to (f);
\draw [thick] (c) to (e);
\draw [thick] (d) to (f);
\draw [thick] (e) to (g);
\draw [thick] (f) to (h);
\draw [thick] (g) to (h);
\draw [thick] (g) -- ++(45:0.5);
\draw [thick] (h) -- ++(-45:0.5);

\draw [thick] (aLeft) to (bLeft);
\draw [thick] (aLeft) to (cLeft);
\draw [thick] (bLeft) to (dLeft);
\draw [thick] (cLeft) to (dLeft);
\draw [thick] (eLeft) to (fLeft);
\draw [thick] (cLeft) to (eLeft);
\draw [thick] (dLeft) to (fLeft);
\draw [thick] (eLeft) to (gLeft);
\draw [thick] (fLeft) to (hLeft);
\draw [thick,bend left] (gLeft) to (hLeft);
\draw [thick,bend left] (hLeft) to (gLeft);

\end{tikzpicture}

%% file: H.tex
\ifx \tikzscale \undefined
   \def \tikzscale {1}
\fi%
\ifx \blockparameter \undefined
   \def \blockparameter {2}%
\fi%
\tikzset{%
  sigma1style/.prefix style={thick,black},%
  connectorstyle/.prefix style={thick,gray!50},%
  unset2factorstyle/.prefix style={thick,blue,densely dashed}%
}
\begin{tikzpicture}[scale=\tikzscale,baseline=0.5]

\pgfmathtruncatemacro\blockparameterminus{\blockparameter-1}
\pgfmathsetmacro\vertexscale{0.5*\tikzscale}
\pgfmathsetmacro\bbRightX{2*\blockparameter-0.5}

\useasboundingbox (-0.5,-0.5) rectangle (\bbRightX,1.5);

\foreach \i in {1,...,\blockparameter}{%
   \pgfmathsetmacro\leftX{(\i-1)*2}
   \pgfmathsetmacro\rightX{(\i-1)*2+1}
   \node [fill,circle, scale=\vertexscale] (a\i) at (\leftX,0) {};
   \node [fill,circle, scale=\vertexscale] (b\i) at (\rightX,0) {};
   \node [fill,circle, scale=\vertexscale] (c\i) at (\leftX,1) {};
   \node [fill,circle, scale=\vertexscale] (d\i) at (\rightX,1) {};
   \draw [unset2factorstyle] (a\i) to (b\i);
   \draw [unset2factorstyle] (a\i) to (c\i);
   \draw [unset2factorstyle] (d\i) to (b\i);
   \draw [unset2factorstyle] (d\i) to (c\i);
}
\ifnum\blockparameter>1
\foreach \i in {1,...,\blockparameterminus}{%
   \pgfmathtruncatemacro\iNext{\i+1}
   \draw [sigma1style] (b\i) to (a\iNext);
   \draw [sigma1style] (d\i) to (c\iNext);
}
\fi
\draw [connectorstyle] (a1) -- ++(210:0.5);
\draw [connectorstyle] (c1) -- ++(150:0.5);
\draw [connectorstyle] (b\blockparameter) -- ++(330:0.5);
\draw [connectorstyle] (d\blockparameter) -- ++(30:0.5);
\end{tikzpicture}

%% file: LH.tex
\ifx \tikzscale \undefined
   \def \tikzscale {1}
\fi%
\ifx \blockparameter \undefined
   \def \blockparameter {2}%
\fi%
\tikzset{%
  sigma1style/.prefix style={thick,black},%
  connectorstyle/.prefix style={thick,gray!50},%
  unset2factorstyle/.prefix style={thick,blue,densely dashed}%
}
\begin{tikzpicture}[scale=\tikzscale,baseline=0.5]

\pgfmathtruncatemacro\blockparameterminus{\blockparameter-1}
\pgfmathsetmacro\vertexscale{0.5*\tikzscale}
\pgfmathsetmacro\bbRightX{2*\blockparameter-0.5}

\useasboundingbox (-0.5,-0.5) rectangle (\bbRightX,1.5);

\foreach \i in {1,...,\blockparameter}{%
   \pgfmathsetmacro\leftX{(\i-1)*2}
   \pgfmathsetmacro\rightX{(\i-1)*2+1}
   \node [fill,circle, scale=\vertexscale] (a\i) at (\leftX,0) {};
   \node [fill,circle, scale=\vertexscale] (b\i) at (\rightX,0) {};
   \node [fill,circle, scale=\vertexscale] (c\i) at (\leftX,1) {};
   \node [fill,circle, scale=\vertexscale] (d\i) at (\rightX,1) {};
   \draw [unset2factorstyle] (a\i) to (b\i);
   \draw [unset2factorstyle] (a\i) to (c\i);
   \draw [unset2factorstyle] (d\i) to (b\i);
   \draw [unset2factorstyle] (d\i) to (c\i);
}
\ifnum\blockparameter>1
\foreach \i in {1,...,\blockparameterminus}{%
   \pgfmathtruncatemacro\iNext{\i+1}
   \draw [sigma1style] (b\i) to (a\iNext);
   \draw [sigma1style] (d\i) to (c\iNext);
}
\fi
\draw [connectorstyle] (a1) -- ++(210:0.5);
\draw [sigma1style] (c1) -- ++(150:0.5);
\draw [connectorstyle] (b\blockparameter) -- ++(330:0.5);
\draw [connectorstyle] (d\blockparameter) -- ++(30:0.5);

\end{tikzpicture}

%% file: DLH.tex
\ifx \tikzscale \undefined
   \def \tikzscale {1}
\fi%
\ifx \blockparameter \undefined
   \def \blockparameter {2}%
\fi%
\tikzset{%
  sigma1style/.prefix style={thick,black},%
  connectorstyle/.prefix style={thick,gray!50},%
  unset2factorstyle/.prefix style={thick,blue,densely dashed}%
}
\begin{tikzpicture}[scale=\tikzscale,baseline=0.5]
\pgfmathtruncatemacro\blockparameterminus{\blockparameter-1}
\pgfmathsetmacro\vertexscale{0.5*\tikzscale}
\pgfmathsetmacro\bbRightX{2*\blockparameter-0.5}

\useasboundingbox (-0.5,-0.5) rectangle (\bbRightX,1.5);

\foreach \i in {1,...,\blockparameter}{%
   \pgfmathsetmacro\leftX{(\i-1)*2}
   \pgfmathsetmacro\rightX{(\i-1)*2+1}
   \node [fill,circle, scale=\vertexscale] (a\i) at (\leftX,0) {};
   \node [fill,circle, scale=\vertexscale] (b\i) at (\rightX,0) {};
   \node [fill,circle, scale=\vertexscale] (c\i) at (\leftX,1) {};
   \node [fill,circle, scale=\vertexscale] (d\i) at (\rightX,1) {};
   \draw [unset2factorstyle] (a\i) to (b\i);
   \draw [unset2factorstyle] (a\i) to (c\i);
   \draw [unset2factorstyle] (d\i) to (b\i);
   \draw [unset2factorstyle] (d\i) to (c\i);
}
\ifnum\blockparameter>1
\foreach \i in {1,...,\blockparameterminus}{%
   \pgfmathtruncatemacro\iNext{\i+1}
   \draw [sigma1style] (b\i) to (a\iNext);
   \draw [sigma1style] (d\i) to (c\iNext);
}
\fi
\draw [connectorstyle] (a1) -- ++(210:0.5);
\draw [sigma1style] (c1) -- ++(150:0.5);
\draw [sigma1style] (b\blockparameter) -- ++(330:0.5);
\draw [connectorstyle] (d\blockparameter) -- ++(30:0.5);

\end{tikzpicture}

%% file: DC.tex
\ifx \tikzscale \undefined
   \def \tikzscale {1}
\fi%
\ifx \blockparameter \undefined
   \def \blockparameter {2}%
\fi%
\tikzset{%
  sigma1style/.prefix style={thick,black},%
  connectorstyle/.prefix style={thick,gray!50},%
  unset2factorstyle/.prefix style={thick,blue,densely dashed}%
}
\begin{tikzpicture}[scale=\tikzscale,baseline=0.5]
\pgfmathtruncatemacro\blockparameterminus{\blockparameter-1}
\pgfmathsetmacro\vertexscale{0.5*\tikzscale}
\pgfmathsetmacro\bbRightX{2*\blockparameter-0.5}

\useasboundingbox (-0.5,-0.5) rectangle (\bbRightX,1.5);

\foreach \i in {1,...,\blockparameter}{%
   \pgfmathsetmacro\leftX{(\i-1)*2}
   \pgfmathsetmacro\rightX{(\i-1)*2+1}
   \node [fill,circle, scale=\vertexscale] (a\i) at (\leftX,0) {};
   \node [fill,circle, scale=\vertexscale] (b\i) at (\rightX,0) {};
   \node [fill,circle, scale=\vertexscale] (c\i) at (\leftX,1) {};
   \node [fill,circle, scale=\vertexscale] (d\i) at (\rightX,1) {};
   \draw [unset2factorstyle] (a\i) to (b\i);
   \draw [unset2factorstyle] (a\i) to (c\i);
   \draw [unset2factorstyle] (d\i) to (b\i);
   \draw [unset2factorstyle] (d\i) to (c\i);
}
\ifnum\blockparameter>1
\foreach \i in {1,...,\blockparameterminus}{%
   \pgfmathtruncatemacro\iNext{\i+1}
   \draw [sigma1style] (b\i) to (a\iNext);
   \draw [sigma1style] (d\i) to (c\iNext);
}
\fi
\draw [connectorstyle] (a1) -- ++(210:0.5);
\draw [connectorstyle] (d\blockparameter) -- ++(30:0.5);

\draw [sigma1style] (c1) -- (b\blockparameter);

\end{tikzpicture}

%% file: DHB.tex
\ifx \tikzscale \undefined
   \def \tikzscale {1}
\fi%
\ifx \blockparameter \undefined
   \def \blockparameter {2}%
\fi%
\tikzset{%
  sigma1style/.prefix style={thick,black},%
  connectorstyle/.prefix style={thick,gray!50},%
  unset2factorstyle/.prefix style={thick,blue,densely dashed}%
}
\begin{tikzpicture}[scale=\tikzscale,baseline=0.5]

\pgfmathtruncatemacro\blockparameterminus{\blockparameter-1}
\pgfmathsetmacro\vertexscale{0.5*\tikzscale}
\pgfmathsetmacro\bbRightX{2*\blockparameter-0.5}

\useasboundingbox (-0.5,-0.5) rectangle (\bbRightX,1.5);

\foreach \i in {1,...,\blockparameter}{%
   \pgfmathsetmacro\leftX{(\i-1)*2}
   \pgfmathsetmacro\rightX{(\i-1)*2+1}
   \node [fill,circle, scale=\vertexscale] (a\i) at (\leftX,0) {};
   \node [fill,circle, scale=\vertexscale] (b\i) at (\rightX,0) {};
   \node [fill,circle, scale=\vertexscale] (c\i) at (\leftX,1) {};
   \node [fill,circle, scale=\vertexscale] (d\i) at (\rightX,1) {};
   \draw [unset2factorstyle] (a\i) to (b\i);
   \draw [unset2factorstyle] (a\i) to (c\i);
   \draw [unset2factorstyle] (d\i) to (b\i);
   \draw [unset2factorstyle] (d\i) to (c\i);
}
\ifnum\blockparameter>1
\foreach \i in {1,...,\blockparameterminus}{%
   \pgfmathtruncatemacro\iNext{\i+1}
   \draw [sigma1style] (b\i) to (a\iNext);
   \draw [sigma1style] (d\i) to (c\iNext);
}
\fi
\draw [connectorstyle] (a1) -- ++(210:0.5);
\draw [connectorstyle] (c1) -- ++(150:0.5);
\draw [sigma1style,bend right] (b\blockparameter) to (d\blockparameter);

\end{tikzpicture}

%% file: OHB.tex
\ifx \tikzscale \undefined
   \def \tikzscale {1}
\fi%
\ifx \blockparameter \undefined
   \def \blockparameter {2}%
\fi%
\tikzset{%
  sigma1style/.prefix style={thick,black},%
  connectorstyle/.prefix style={thick,gray!50},%
  unset2factorstyle/.prefix style={thick,blue,densely dashed}%
}
\begin{tikzpicture}[scale=\tikzscale,baseline=0.5]

\pgfmathtruncatemacro\blockparameterminus{\blockparameter-1}
\pgfmathsetmacro\vertexscale{0.5*\tikzscale}
\pgfmathsetmacro\bbRightX{2*\blockparameter-0.5}

\useasboundingbox (-0.5,-0.5) rectangle (\bbRightX,1.5);

\foreach \i in {1,...,\blockparameter}{%
   \pgfmathsetmacro\leftX{(\i-1)*2}
   \pgfmathsetmacro\rightX{(\i-1)*2+1}
   \node [fill,circle, scale=\vertexscale] (a\i) at (\leftX,0) {};
   \node [fill,circle, scale=\vertexscale] (b\i) at (\rightX,0) {};
   \node [fill,circle, scale=\vertexscale] (c\i) at (\leftX,1) {};
   \node [fill,circle, scale=\vertexscale] (d\i) at (\rightX,1) {};
   \draw [unset2factorstyle] (a\i) to (b\i);
   \draw [unset2factorstyle] (a\i) to (c\i);
   \draw [unset2factorstyle] (d\i) to (b\i);
   \draw [unset2factorstyle] (d\i) to (c\i);
}
\ifnum\blockparameter>1
\foreach \i in {1,...,\blockparameterminus}{%
   \pgfmathtruncatemacro\iNext{\i+1}
   \draw [sigma1style] (b\i) to (a\iNext);
   \draw [sigma1style] (d\i) to (c\iNext);
}
\fi
\draw [connectorstyle] (a1) -- ++(210:0.5);
\draw [connectorstyle] (c1) -- ++(150:0.5);
\draw [sigma1style] (b\blockparameter) -- ++(330:0.5);
\draw [sigma1style] (d\blockparameter) -- ++(30:0.5);
\end{tikzpicture}

%% file: DLB.tex
\ifx \tikzscale \undefined
   \def \tikzscale {1}
\fi%
\ifx \blockparameter \undefined
   \def \blockparameter {2}%
\fi%
\tikzset{%
  sigma1style/.prefix style={thick,black},%
  connectorstyle/.prefix style={thick,gray!50},%
  unset2factorstyle/.prefix style={thick,blue,densely dashed}%
}
\begin{tikzpicture}[scale=\tikzscale,baseline=0.5]

\pgfmathtruncatemacro\blockparameterminus{\blockparameter-1}
\pgfmathsetmacro\vertexscale{0.5*\tikzscale}
\pgfmathsetmacro\bbRightX{2*\blockparameter-0.5}

\useasboundingbox (-0.5,-0.5) rectangle (\bbRightX,1.5);

\foreach \i in {1,...,\blockparameter}{%
   \pgfmathsetmacro\leftX{(\i-1)*2}
   \pgfmathsetmacro\rightX{(\i-1)*2+1}
   \node [fill,circle, scale=\vertexscale] (a\i) at (\leftX,0) {};
   \node [fill,circle, scale=\vertexscale] (b\i) at (\rightX,0) {};
   \node [fill,circle, scale=\vertexscale] (c\i) at (\leftX,1) {};
   \node [fill,circle, scale=\vertexscale] (d\i) at (\rightX,1) {};
   \draw [unset2factorstyle] (a\i) to (b\i);
   \draw [unset2factorstyle] (a\i) to (c\i);
   \draw [unset2factorstyle] (d\i) to (b\i);
   \draw [unset2factorstyle] (d\i) to (c\i);
}
\ifnum\blockparameter>1
\foreach \i in {1,...,\blockparameterminus}{%
   \pgfmathtruncatemacro\iNext{\i+1}
   \draw [sigma1style] (b\i) to (a\iNext);
   \draw [sigma1style] (d\i) to (c\iNext);
}
\fi
\draw [connectorstyle] (a1) -- ++(210:0.5);
\draw [connectorstyle] (b\blockparameter) -- ++(330:0.5);

\draw [sigma1style, bend left=30] (c1) to (d\blockparameter);

\end{tikzpicture}

%% file: OLB.tex
\ifx \tikzscale \undefined
   \def \tikzscale {1}
\fi%
\ifx \blockparameter \undefined
   \def \blockparameter {2}%
\fi%
\tikzset{%
  sigma1style/.prefix style={thick,black},%
  connectorstyle/.prefix style={thick,gray!50},%
  unset2factorstyle/.prefix style={thick,blue,densely dashed}%
}
\begin{tikzpicture}[scale=\tikzscale,baseline=0.5]

\pgfmathtruncatemacro\blockparameterminus{\blockparameter-1}
\pgfmathsetmacro\vertexscale{0.5*\tikzscale}
\pgfmathsetmacro\bbRightX{2*\blockparameter-0.5}

\useasboundingbox (-0.5,-0.5) rectangle (\bbRightX,1.5);

\foreach \i in {1,...,\blockparameter}{%
   \pgfmathsetmacro\leftX{(\i-1)*2}
   \pgfmathsetmacro\rightX{(\i-1)*2+1}
   \node [fill,circle, scale=\vertexscale] (a\i) at (\leftX,0) {};
   \node [fill,circle, scale=\vertexscale] (b\i) at (\rightX,0) {};
   \node [fill,circle, scale=\vertexscale] (c\i) at (\leftX,1) {};
   \node [fill,circle, scale=\vertexscale] (d\i) at (\rightX,1) {};
   \draw [unset2factorstyle] (a\i) to (b\i);
   \draw [unset2factorstyle] (a\i) to (c\i);
   \draw [unset2factorstyle] (d\i) to (b\i);
   \draw [unset2factorstyle] (d\i) to (c\i);
}
\ifnum\blockparameter>1
\foreach \i in {1,...,\blockparameterminus}{%
   \pgfmathtruncatemacro\iNext{\i+1}
   \draw [sigma1style] (b\i) to (a\iNext);
   \draw [sigma1style] (d\i) to (c\iNext);
}
\fi
\draw [connectorstyle] (a1) -- ++(210:0.5);
\draw [sigma1style] (c1) -- ++(150:0.5);
\draw [connectorstyle] (b\blockparameter) -- ++(330:0.5);
\draw [sigma1style] (d\blockparameter) -- ++(30:0.5);

\end{tikzpicture}

%% file: LDC.tex
\ifx \tikzscale \undefined
   \def \tikzscale {1}
\fi%
\ifx \blockparameter \undefined
   \def \blockparameter {2}%
\fi%
\tikzset{%
  sigma1style/.prefix style={thick,black},%
  connectorstyle/.prefix style={thick,gray!50},%
  unset2factorstyle/.prefix style={thick,blue,densely dashed}%
}
\begin{tikzpicture}[scale=\tikzscale,baseline=0.5]
\pgfmathtruncatemacro\blockparameterminus{\blockparameter-1}
\pgfmathsetmacro\vertexscale{0.5*\tikzscale}
\pgfmathsetmacro\bbRightX{2*\blockparameter-0.5}

\useasboundingbox (-0.5,-0.5) rectangle (\bbRightX,1.5);

\foreach \i in {1,...,\blockparameter}{%
   \pgfmathsetmacro\leftX{(\i-1)*2}
   \pgfmathsetmacro\rightX{(\i-1)*2+1}
   \node [fill,circle, scale=\vertexscale] (a\i) at (\leftX,0) {};
   \node [fill,circle, scale=\vertexscale] (b\i) at (\rightX,0) {};
   \node [fill,circle, scale=\vertexscale] (c\i) at (\leftX,1) {};
   \node [fill,circle, scale=\vertexscale] (d\i) at (\rightX,1) {};
   \draw [unset2factorstyle] (a\i) to (b\i);
   \draw [unset2factorstyle] (a\i) to (c\i);
   \draw [unset2factorstyle] (d\i) to (b\i);
   \draw [unset2factorstyle] (d\i) to (c\i);
}
\ifnum\blockparameter>1
\foreach \i in {1,...,\blockparameterminus}{%
   \pgfmathtruncatemacro\iNext{\i+1}
   \draw [sigma1style] (b\i) to (a\iNext);
   \draw [sigma1style] (d\i) to (c\iNext);
}
\fi
\draw [connectorstyle] (a1) -- ++(210:0.5);
\draw [sigma1style] (d\blockparameter) -- ++(30:0.5);

\draw [sigma1style] (c1) -- (b\blockparameter);

\end{tikzpicture}

%% file: LDHB.tex
\ifx \tikzscale \undefined
   \def \tikzscale {1}
\fi%
\ifx \blockparameter \undefined
   \def \blockparameter {2}%
\fi%
\tikzset{%
  sigma1style/.prefix style={thick,black},%
  connectorstyle/.prefix style={thick,gray!50},%
  unset2factorstyle/.prefix style={thick,blue,densely dashed}%
}
\begin{tikzpicture}[scale=\tikzscale,baseline=0.5]

\pgfmathtruncatemacro\blockparameterminus{\blockparameter-1}
\pgfmathsetmacro\vertexscale{0.5*\tikzscale}
\pgfmathsetmacro\bbRightX{2*\blockparameter-0.5}

\useasboundingbox (-0.5,-0.5) rectangle (\bbRightX,1.5);

\foreach \i in {1,...,\blockparameter}{%
   \pgfmathsetmacro\leftX{(\i-1)*2}
   \pgfmathsetmacro\rightX{(\i-1)*2+1}
   \node [fill,circle, scale=\vertexscale] (a\i) at (\leftX,0) {};
   \node [fill,circle, scale=\vertexscale] (b\i) at (\rightX,0) {};
   \node [fill,circle, scale=\vertexscale] (c\i) at (\leftX,1) {};
   \node [fill,circle, scale=\vertexscale] (d\i) at (\rightX,1) {};
   \draw [unset2factorstyle] (a\i) to (b\i);
   \draw [unset2factorstyle] (a\i) to (c\i);
   \draw [unset2factorstyle] (d\i) to (b\i);
   \draw [unset2factorstyle] (d\i) to (c\i);
}
\ifnum\blockparameter>1
\foreach \i in {1,...,\blockparameterminus}{%
   \pgfmathtruncatemacro\iNext{\i+1}
   \draw [sigma1style] (b\i) to (a\iNext);
   \draw [sigma1style] (d\i) to (c\iNext);
}
\fi
\draw [sigma1style] (a1) -- ++(210:0.5);
\draw [connectorstyle] (c1) -- ++(150:0.5);
\draw [sigma1style,bend right] (b\blockparameter) to (d\blockparameter);

\end{tikzpicture}

%% file: LOHB.tex
\ifx \tikzscale \undefined
   \def \tikzscale {1}
\fi%
\ifx \blockparameter \undefined
   \def \blockparameter {2}%
\fi%
\tikzset{%
  sigma1style/.prefix style={thick,black},%
  connectorstyle/.prefix style={thick,gray!50},%
  unset2factorstyle/.prefix style={thick,blue,densely dashed}%
}
\begin{tikzpicture}[scale=\tikzscale,baseline=0.5]

\pgfmathtruncatemacro\blockparameterminus{\blockparameter-1}
\pgfmathsetmacro\vertexscale{0.5*\tikzscale}
\pgfmathsetmacro\bbRightX{2*\blockparameter-0.5}

\useasboundingbox (-0.5,-0.5) rectangle (\bbRightX,1.5);

\foreach \i in {1,...,\blockparameter}{%
   \pgfmathsetmacro\leftX{(\i-1)*2}
   \pgfmathsetmacro\rightX{(\i-1)*2+1}
   \node [fill,circle, scale=\vertexscale] (a\i) at (\leftX,0) {};
   \node [fill,circle, scale=\vertexscale] (b\i) at (\rightX,0) {};
   \node [fill,circle, scale=\vertexscale] (c\i) at (\leftX,1) {};
   \node [fill,circle, scale=\vertexscale] (d\i) at (\rightX,1) {};
   \draw [unset2factorstyle] (a\i) to (b\i);
   \draw [unset2factorstyle] (a\i) to (c\i);
   \draw [unset2factorstyle] (d\i) to (b\i);
   \draw [unset2factorstyle] (d\i) to (c\i);
}
\ifnum\blockparameter>1
\foreach \i in {1,...,\blockparameterminus}{%
   \pgfmathtruncatemacro\iNext{\i+1}
   \draw [sigma1style] (b\i) to (a\iNext);
   \draw [sigma1style] (d\i) to (c\iNext);
}
\fi
\draw [sigma1style] (a1) -- ++(210:0.5);
\draw [connectorstyle] (c1) -- ++(150:0.5);
\draw [sigma1style] (b\blockparameter) -- ++(330:0.5);
\draw [sigma1style] (d\blockparameter) -- ++(30:0.5);
\end{tikzpicture}

%% file: LDLB.tex
\ifx \tikzscale \undefined
   \def \tikzscale {1}
\fi%
\ifx \blockparameter \undefined
   \def \blockparameter {2}%
\fi%
\tikzset{%
  sigma1style/.prefix style={thick,black},%
  connectorstyle/.prefix style={thick,gray!50},%
  unset2factorstyle/.prefix style={thick,blue,densely dashed}%
}
\begin{tikzpicture}[scale=\tikzscale,baseline=0.5]

\pgfmathtruncatemacro\blockparameterminus{\blockparameter-1}
\pgfmathsetmacro\vertexscale{0.5*\tikzscale}
\pgfmathsetmacro\bbRightX{2*\blockparameter-0.5}

\useasboundingbox (-0.5,-0.5) rectangle (\bbRightX,1.5);

\foreach \i in {1,...,\blockparameter}{%
   \pgfmathsetmacro\leftX{(\i-1)*2}
   \pgfmathsetmacro\rightX{(\i-1)*2+1}
   \node [fill,circle, scale=\vertexscale] (a\i) at (\leftX,0) {};
   \node [fill,circle, scale=\vertexscale] (b\i) at (\rightX,0) {};
   \node [fill,circle, scale=\vertexscale] (c\i) at (\leftX,1) {};
   \node [fill,circle, scale=\vertexscale] (d\i) at (\rightX,1) {};
   \draw [unset2factorstyle] (a\i) to (b\i);
   \draw [unset2factorstyle] (a\i) to (c\i);
   \draw [unset2factorstyle] (d\i) to (b\i);
   \draw [unset2factorstyle] (d\i) to (c\i);
}
\ifnum\blockparameter>1
\foreach \i in {1,...,\blockparameterminus}{%
   \pgfmathtruncatemacro\iNext{\i+1}
   \draw [sigma1style] (b\i) to (a\iNext);
   \draw [sigma1style] (d\i) to (c\iNext);
}
\fi
\draw [connectorstyle] (a1) -- ++(210:0.5);
\draw [sigma1style] (b\blockparameter) -- ++(330:0.5);

\draw [sigma1style, bend left=30] (c1) to (d\blockparameter);
\end{tikzpicture}

%% file: DLDC.tex
\ifx \tikzscale \undefined
   \def \tikzscale {1}
\fi%
\ifx \blockparameter \undefined
   \def \blockparameter {2}%
\fi%
\tikzset{%
  sigma1style/.prefix style={thick,black},%
  connectorstyle/.prefix style={thick,gray!50},%
  unset2factorstyle/.prefix style={thick,blue,densely dashed}%
}
\begin{tikzpicture}[scale=\tikzscale,baseline=0.5]
\pgfmathtruncatemacro\blockparameterminus{\blockparameter-1}
\pgfmathsetmacro\vertexscale{0.5*\tikzscale}
\pgfmathsetmacro\bbRightX{2*\blockparameter-0.5}

\useasboundingbox (-0.5,-0.5) rectangle (\bbRightX,1.5);

\foreach \i in {1,...,\blockparameter}{%
   \pgfmathsetmacro\leftX{(\i-1)*2}
   \pgfmathsetmacro\rightX{(\i-1)*2+1}
   \node [fill,circle, scale=\vertexscale] (a\i) at (\leftX,0) {};
   \node [fill,circle, scale=\vertexscale] (b\i) at (\rightX,0) {};
   \node [fill,circle, scale=\vertexscale] (c\i) at (\leftX,1) {};
   \node [fill,circle, scale=\vertexscale] (d\i) at (\rightX,1) {};
   \draw [unset2factorstyle] (a\i) to (b\i);
   \draw [unset2factorstyle] (a\i) to (c\i);
   \draw [unset2factorstyle] (d\i) to (b\i);
   \draw [unset2factorstyle] (d\i) to (c\i);
}
\ifnum\blockparameter>1
\foreach \i in {1,...,\blockparameterminus}{%
   \pgfmathtruncatemacro\iNext{\i+1}
   \draw [sigma1style] (b\i) to (a\iNext);
   \draw [sigma1style] (d\i) to (c\iNext);
}
\fi
\draw [sigma1style] (a1) -- ++(210:0.5);
\draw [sigma1style] (d\blockparameter) -- ++(30:0.5);

\draw [sigma1style] (c1) -- (b\blockparameter);
\end{tikzpicture}

%% file: CLH.tex
\ifx \tikzscale \undefined
   \def \tikzscale {1}
\fi%
\ifx \blockparameter \undefined
   \def \blockparameter {2}%
\fi%
\tikzset{%
  sigma1style/.prefix style={thick,black},%
  connectorstyle/.prefix style={thick,gray!50},%
  unset2factorstyle/.prefix style={thick,blue,densely dashed}%
}
\begin{tikzpicture}[scale=\tikzscale,baseline=0.5]

\pgfmathtruncatemacro\blockparameterminus{\blockparameter-1}
\pgfmathsetmacro\vertexscale{0.5*\tikzscale}
\pgfmathsetmacro\bbRightX{2*\blockparameter-0.5}

\useasboundingbox (-0.5,-0.5) rectangle (\bbRightX,1.5);

\foreach \i in {1,...,\blockparameter}{%
   \pgfmathsetmacro\leftX{(\i-1)*2}
   \pgfmathsetmacro\rightX{(\i-1)*2+1}
   \node [fill,circle, scale=\vertexscale] (a\i) at (\leftX,0) {};
   \node [fill,circle, scale=\vertexscale] (b\i) at (\rightX,0) {};
   \node [fill,circle, scale=\vertexscale] (c\i) at (\leftX,1) {};
   \node [fill,circle, scale=\vertexscale] (d\i) at (\rightX,1) {};
   \draw [unset2factorstyle] (a\i) to (b\i);
   \draw [unset2factorstyle] (a\i) to (c\i);
   \draw [unset2factorstyle] (d\i) to (b\i);
   \draw [unset2factorstyle] (d\i) to (c\i);
}
\ifnum\blockparameter>1
\foreach \i in {1,...,\blockparameterminus}{%
   \pgfmathtruncatemacro\iNext{\i+1}
   \draw [sigma1style] (b\i) to (a\iNext);
   \draw [sigma1style] (d\i) to (c\iNext);
}
\fi
\draw [sigma1style] (a1) -- ++(210:0.5);
\draw [sigma1style] (c1) -- ++(150:0.5);
\draw [sigma1style] (b\blockparameter) -- ++(330:0.5);
\draw [sigma1style] (d\blockparameter) -- ++(30:0.5);

\end{tikzpicture}

%% file: DDHB.tex
\ifx \tikzscale \undefined
   \def \tikzscale {1}
\fi%
\ifx \blockparameter \undefined
   \def \blockparameter {2}%
\fi%
\tikzset{%
  sigma1style/.prefix style={thick,black},%
  connectorstyle/.prefix style={thick,gray!50},%
  unset2factorstyle/.prefix style={thick,blue,densely dashed}%
}
\begin{tikzpicture}[scale=\tikzscale,baseline=0.5]

\pgfmathtruncatemacro\blockparameterminus{\blockparameter-1}
\pgfmathsetmacro\vertexscale{0.5*\tikzscale}
\pgfmathsetmacro\bbRightX{2*\blockparameter-0.5}

\useasboundingbox (-0.5,-0.5) rectangle (\bbRightX,1.5);

\foreach \i in {1,...,\blockparameter}{%
   \pgfmathsetmacro\leftX{(\i-1)*2}
   \pgfmathsetmacro\rightX{(\i-1)*2+1}
   \node [fill,circle, scale=\vertexscale] (a\i) at (\leftX,0) {};
   \node [fill,circle, scale=\vertexscale] (b\i) at (\rightX,0) {};
   \node [fill,circle, scale=\vertexscale] (c\i) at (\leftX,1) {};
   \node [fill,circle, scale=\vertexscale] (d\i) at (\rightX,1) {};
   \draw [unset2factorstyle] (a\i) to (b\i);
   \draw [unset2factorstyle] (a\i) to (c\i);
   \draw [unset2factorstyle] (d\i) to (b\i);
   \draw [unset2factorstyle] (d\i) to (c\i);
}
\ifnum\blockparameter>1
\foreach \i in {1,...,\blockparameterminus}{%
   \pgfmathtruncatemacro\iNext{\i+1}
   \draw [sigma1style] (b\i) to (a\iNext);
   \draw [sigma1style] (d\i) to (c\iNext);
}
\fi

\draw [sigma1style,bend left] (a1) to (c1);

\draw [sigma1style,bend right] (b\blockparameter) to (d\blockparameter);

\end{tikzpicture}

%% file: DLDHB.tex
\ifx \tikzscale \undefined
   \def \tikzscale {1}
\fi%
\ifx \blockparameter \undefined
   \def \blockparameter {2}%
\fi%
\tikzset{%
  sigma1style/.prefix style={thick,black},%
  connectorstyle/.prefix style={thick,gray!50},%
  unset2factorstyle/.prefix style={thick,blue,densely dashed}%
}
\begin{tikzpicture}[scale=\tikzscale,baseline=0.5]

\pgfmathtruncatemacro\blockparameterminus{\blockparameter-1}
\pgfmathsetmacro\vertexscale{0.5*\tikzscale}
\pgfmathsetmacro\bbRightX{2*\blockparameter-0.5}

\useasboundingbox (-0.5,-0.5) rectangle (\bbRightX,1.5);

\foreach \i in {1,...,\blockparameter}{%
   \pgfmathsetmacro\leftX{(\i-1)*2}
   \pgfmathsetmacro\rightX{(\i-1)*2+1}
   \node [fill,circle, scale=\vertexscale] (a\i) at (\leftX,0) {};
   \node [fill,circle, scale=\vertexscale] (b\i) at (\rightX,0) {};
   \node [fill,circle, scale=\vertexscale] (c\i) at (\leftX,1) {};
   \node [fill,circle, scale=\vertexscale] (d\i) at (\rightX,1) {};
   \draw [unset2factorstyle] (a\i) to (b\i);
   \draw [unset2factorstyle] (a\i) to (c\i);
   \draw [unset2factorstyle] (d\i) to (b\i);
   \draw [unset2factorstyle] (d\i) to (c\i);
}
\ifnum\blockparameter>1
\foreach \i in {1,...,\blockparameterminus}{%
   \pgfmathtruncatemacro\iNext{\i+1}
   \draw [sigma1style] (b\i) to (a\iNext);
   \draw [sigma1style] (d\i) to (c\iNext);
}
\fi
\draw [sigma1style] (a1) -- ++(210:0.5);
\draw [sigma1style] (c1) -- ++(150:0.5);
\draw [sigma1style,bend right] (b\blockparameter) to (d\blockparameter);

\end{tikzpicture}

%% file: DLDLB.tex
\ifx \tikzscale \undefined
   \def \tikzscale {1}
\fi%
\ifx \blockparameter \undefined
   \def \blockparameter {2}%
\fi%
\tikzset{%
  sigma1style/.prefix style={thick,black},%
  connectorstyle/.prefix style={thick,gray!50},%
  unset2factorstyle/.prefix style={thick,blue,densely dashed}%
}
\begin{tikzpicture}[scale=\tikzscale,baseline=0.5]

\pgfmathtruncatemacro\blockparameterminus{\blockparameter-1}
\pgfmathsetmacro\vertexscale{0.5*\tikzscale}
\pgfmathsetmacro\bbRightX{2*\blockparameter-0.5}

\useasboundingbox (-0.5,-0.5) rectangle (\bbRightX,1.5);

\foreach \i in {1,...,\blockparameter}{%
   \pgfmathsetmacro\leftX{(\i-1)*2}
   \pgfmathsetmacro\rightX{(\i-1)*2+1}
   \node [fill,circle, scale=\vertexscale] (a\i) at (\leftX,0) {};
   \node [fill,circle, scale=\vertexscale] (b\i) at (\rightX,0) {};
   \node [fill,circle, scale=\vertexscale] (c\i) at (\leftX,1) {};
   \node [fill,circle, scale=\vertexscale] (d\i) at (\rightX,1) {};
   \draw [unset2factorstyle] (a\i) to (b\i);
   \draw [unset2factorstyle] (a\i) to (c\i);
   \draw [unset2factorstyle] (d\i) to (b\i);
   \draw [unset2factorstyle] (d\i) to (c\i);
}
\ifnum\blockparameter>1
\foreach \i in {1,...,\blockparameterminus}{%
   \pgfmathtruncatemacro\iNext{\i+1}
   \draw [sigma1style] (b\i) to (a\iNext);
   \draw [sigma1style] (d\i) to (c\iNext);
}
\fi
\draw [sigma1style] (a1) -- ++(210:0.5);
\draw [sigma1style] (b\blockparameter) -- ++(330:0.5);

\draw [sigma1style, bend left=30] (c1) to (d\blockparameter);

\end{tikzpicture}

%% file: ML.tex
\ifx \tikzscale \undefined
   \def \tikzscale {1}
\fi%
\ifx \blockparameter \undefined
   \def \blockparameter {2}%
\fi%
\tikzset{%
  sigma1style/.prefix style={thick,black},%
  connectorstyle/.prefix style={thick,gray!50},%
  unset2factorstyle/.prefix style={thick,blue,densely dashed}%
}
\begin{tikzpicture}[scale=\tikzscale,baseline=0.5]

\pgfmathtruncatemacro\blockparameterminus{\blockparameter-1}
\pgfmathsetmacro\vertexscale{0.5*\tikzscale}
\pgfmathsetmacro\bbRightX{2*\blockparameter-0.5}

\useasboundingbox (-0.5,-0.5) rectangle (\bbRightX,1.5);

\foreach \i in {1,...,\blockparameter}{%
   \pgfmathsetmacro\leftX{(\i-1)*2}
   \pgfmathsetmacro\rightX{(\i-1)*2+1}
   \node [fill,circle, scale=\vertexscale] (a\i) at (\leftX,0) {};
   \node [fill,circle, scale=\vertexscale] (b\i) at (\rightX,0) {};
   \node [fill,circle, scale=\vertexscale] (c\i) at (\leftX,1) {};
   \node [fill,circle, scale=\vertexscale] (d\i) at (\rightX,1) {};
   \ifnum\i=1
      \draw [unset2factorstyle] (a\i) to (d\i);
      \draw [unset2factorstyle] (a\i) to (c\i);
      \draw [unset2factorstyle] (d\i) to (b\i);
      \draw [unset2factorstyle] (b\i) to (c\i);
   \else
      \draw [unset2factorstyle] (a\i) to (b\i);
      \draw [unset2factorstyle] (a\i) to (c\i);
      \draw [unset2factorstyle] (d\i) to (b\i);
      \draw [unset2factorstyle] (d\i) to (c\i);
   \fi
}
\ifnum\blockparameter>1
\foreach \i in {1,...,\blockparameterminus}{%
   \pgfmathtruncatemacro\iNext{\i+1}
   \draw [sigma1style] (b\i) to (a\iNext);
   \draw [sigma1style] (d\i) to (c\iNext);
}
\fi

\draw [sigma1style, bend left=30] (c1) to (d\blockparameter);
\draw [sigma1style, bend right=30] (a1) to (b\blockparameter);

\end{tikzpicture}

%% file: P.tex
\ifx \tikzscale \undefined
   \def \tikzscale {1}
\fi%
\ifx \blockparameter \undefined
   \def \blockparameter {2}%
\fi%
\tikzset{%
  sigma1style/.prefix style={thick,black},%
  connectorstyle/.prefix style={thick,gray!50},%
  unset2factorstyle/.prefix style={thick,blue,densely dashed}%
}
\begin{tikzpicture}[scale=\tikzscale,baseline=0.5]

\pgfmathtruncatemacro\blockparameterminus{\blockparameter-1}
\pgfmathsetmacro\vertexscale{0.5*\tikzscale}
\pgfmathsetmacro\bbRightX{2*\blockparameter-0.5}

\useasboundingbox (-0.5,-0.5) rectangle (\bbRightX,1.5);

\foreach \i in {1,...,\blockparameter}{%
   \pgfmathsetmacro\leftX{(\i-1)*2}
   \pgfmathsetmacro\rightX{(\i-1)*2+1}
   \node [fill,circle, scale=\vertexscale] (a\i) at (\leftX,0) {};
   \node [fill,circle, scale=\vertexscale] (b\i) at (\rightX,0) {};
   \node [fill,circle, scale=\vertexscale] (c\i) at (\leftX,1) {};
   \node [fill,circle, scale=\vertexscale] (d\i) at (\rightX,1) {};
   \draw [unset2factorstyle] (a\i) to (b\i);
   \draw [unset2factorstyle] (a\i) to (c\i);
   \draw [unset2factorstyle] (d\i) to (b\i);
   \draw [unset2factorstyle] (d\i) to (c\i);
}
\ifnum\blockparameter>1
\foreach \i in {1,...,\blockparameterminus}{%
   \pgfmathtruncatemacro\iNext{\i+1}
   \draw [sigma1style] (b\i) to (a\iNext);
   \draw [sigma1style] (d\i) to (c\iNext);
}
\fi

\draw [sigma1style, bend left=30] (c1) to (d\blockparameter);
\draw [sigma1style, bend right=30] (a1) to (b\blockparameter);

\end{tikzpicture}

%% file: PC.tex
\ifx \tikzscale \undefined
   \def \tikzscale {1}
\fi%
\ifx \blockparameter \undefined
   \def \blockparameter {2}%
\fi%
\tikzset{%
  sigma1style/.prefix style={thick,black},%
  connectorstyle/.prefix style={thick,gray!50},%
  unset2factorstyle/.prefix style={thick,blue,densely dashed}%
}
\begin{tikzpicture}[scale=\tikzscale,baseline=-2]

\pgfmathtruncatemacro\blockparameterminus{\blockparameter-1}
\pgfmathsetmacro\vertexscale{0.5*\tikzscale}
\pgfmathsetmacro\bbRightX{2*\blockparameter}

\useasboundingbox (-0.5,-0.5) rectangle (\bbRightX,1.5);

\foreach \i in {1,...,\blockparameter}{%
   \pgfmathsetmacro\leftX{(\i-1)*2}
   \pgfmathsetmacro\rightX{(\i-1)*2+1}
   \node [fill,circle, scale=\vertexscale] (a\i) at (\leftX,0) {};
   \node [fill,circle, scale=\vertexscale] (b\i) at (\rightX,0) {};
   \draw [unset2factorstyle,bend left] (a\i) to (b\i);
   \draw [unset2factorstyle,bend right] (a\i) to (b\i);
}
\ifnum\blockparameter>1
\foreach \i in {1,...,\blockparameterminus}{%
   \pgfmathtruncatemacro\iNext{\i+1}
   \draw [sigma1style] (b\i) to (a\iNext);
}
\fi
\draw [connectorstyle] (a1) -- ++(180:0.5);
\draw [connectorstyle] (b\blockparameter) -- ++(0:0.5);
\end{tikzpicture}

%% file: DLPC.tex
\ifx \tikzscale \undefined
   \def \tikzscale {1}
\fi%
\ifx \blockparameter \undefined
   \def \blockparameter {2}%
\fi%
\tikzset{%
  sigma1style/.prefix style={thick,black},%
  connectorstyle/.prefix style={thick,gray!50},%
  unset2factorstyle/.prefix style={thick,blue,densely dashed}%
}
\begin{tikzpicture}[scale=\tikzscale,baseline=-2]

\pgfmathtruncatemacro\blockparameterminus{\blockparameter-1}
\pgfmathsetmacro\vertexscale{0.5*\tikzscale}
\pgfmathsetmacro\bbRightX{2*\blockparameter}

\useasboundingbox (-0.5,-0.5) rectangle (\bbRightX,1.5);

\foreach \i in {1,...,\blockparameter}{%
   \pgfmathsetmacro\leftX{(\i-1)*2}
   \pgfmathsetmacro\rightX{(\i-1)*2+1}
   \node [fill,circle, scale=\vertexscale] (a\i) at (\leftX,0) {};
   \node [fill,circle, scale=\vertexscale] (b\i) at (\rightX,0) {};
   \draw [unset2factorstyle,bend left] (a\i) to (b\i);
   \draw [unset2factorstyle,bend right] (a\i) to (b\i);
}
\ifnum\blockparameter>1
\foreach \i in {1,...,\blockparameterminus}{%
   \pgfmathtruncatemacro\iNext{\i+1}
   \draw [sigma1style] (b\i) to (a\iNext);
}
\fi
\draw [sigma1style] (a1) -- ++(180:0.5);
\draw [sigma1style] (b\blockparameter) -- ++(0:0.5);

\end{tikzpicture}

%% file: LPC.tex
\ifx \tikzscale \undefined
   \def \tikzscale {1}
\fi%
\ifx \blockparameter \undefined
   \def \blockparameter {2}%
\fi%
\tikzset{%
  sigma1style/.prefix style={thick,black},%
  connectorstyle/.prefix style={thick,gray!50},%
  unset2factorstyle/.prefix style={thick,blue,densely dashed}%
}
\begin{tikzpicture}[scale=\tikzscale,baseline=-2]

\pgfmathtruncatemacro\blockparameterminus{\blockparameter-1}
\pgfmathsetmacro\vertexscale{0.5*\tikzscale}
\pgfmathsetmacro\bbRightX{2*\blockparameter}

\useasboundingbox (-0.5,-0.5) rectangle (\bbRightX,1.5);

\foreach \i in {1,...,\blockparameter}{%
   \pgfmathsetmacro\leftX{(\i-1)*2}
   \pgfmathsetmacro\rightX{(\i-1)*2+1}
   \node [fill,circle, scale=\vertexscale] (a\i) at (\leftX,0) {};
   \node [fill,circle, scale=\vertexscale] (b\i) at (\rightX,0) {};
   \draw [unset2factorstyle,bend left] (a\i) to (b\i);
   \draw [unset2factorstyle,bend right] (a\i) to (b\i);
}
\ifnum\blockparameter>1
\foreach \i in {1,...,\blockparameterminus}{%
   \pgfmathtruncatemacro\iNext{\i+1}
   \draw [sigma1style] (b\i) to (a\iNext);
}
\fi
\draw [connectorstyle] (a1) -- ++(180:0.5);
\draw [sigma1style] (b\blockparameter) -- ++(0:0.5);
\end{tikzpicture}

%% file: PN.tex
\ifx \tikzscale \undefined
   \def \tikzscale {1}
\fi%
\ifx \blockparameter \undefined
   \def \blockparameter {2}%
\fi%
\tikzset{%
  sigma1style/.prefix style={thick,black},%
  connectorstyle/.prefix style={thick,gray!50},%
  unset2factorstyle/.prefix style={thick,blue,densely dashed}%
}
\begin{tikzpicture}[scale=\tikzscale,baseline=-2]

\pgfmathsetmacro\vertexscale{0.5*\tikzscale}
\pgfmathsetmacro\necklaceRadius{0.5/sin(90/\blockparameter)}
\pgfmathsetmacro\necklaceAngle{180/\blockparameter}

%\useasboundingbox (-0.5,-0.5) rectangle (0.5,0.5);

\foreach \i in {1,...,\blockparameter}{%
   \pgfmathsetmacro\currentAngleA{(2*(\i-1))*\necklaceAngle}
   \pgfmathsetmacro\currentAngleB{(2*(\i-1)+1)*\necklaceAngle}
   \node [fill,circle, scale=\vertexscale] (a\i) at (\currentAngleA:\necklaceRadius) {};
   \node [fill,circle, scale=\vertexscale] (b\i) at (\currentAngleB:\necklaceRadius) {};
   \draw [unset2factorstyle] (a\i) to [bend left] (b\i);
   \draw [unset2factorstyle] (a\i) to [bend right] (b\i);
   \ifnum\i=\blockparameter
      \pgfmathsetmacro\markAngle{-\necklaceAngle/2}
       \draw [sigma1style] (a1) -- (b\blockparameter);
   \fi
   \ifnum\i>1
      \pgfmathtruncatemacro\iPrev{\i-1}
      \pgfmathsetmacro\markAngle{(2*\i-2.5)*\necklaceAngle}
      \draw [sigma1style] (a\i) -- (b\iPrev);
   \fi
}
\end{tikzpicture}

%% file: BW.tex
\ifx \tikzscale \undefined
   \def \tikzscale {1}
\fi%
\ifx \blockparameter \undefined
   \def \blockparameter {2}%
\fi%
\tikzset{%
  sigma1style/.prefix style={thick,black},%
  connectorstyle/.prefix style={thick,gray!50},%
  unset2factorstyle/.prefix style={thick,blue,densely dashed}%
}
\begin{tikzpicture}[scale=\tikzscale,baseline=-2]

\pgfmathtruncatemacro\blockparameterminus{\blockparameter-1}
\pgfmathsetmacro\vertexscale{0.5*\tikzscale}
\pgfmathsetmacro\bbRightX{2*\blockparameter}

\useasboundingbox (-0.5,-0.5) rectangle (\bbRightX,1.5);

\foreach \i in {1,...,\blockparameter}{%
   \pgfmathsetmacro\leftX{(\i-1)*2}
   \pgfmathsetmacro\rightX{(\i-1)*2+1}
   \node [fill,circle, scale=\vertexscale] (a\i) at (\leftX,0) {};
   \node [fill,circle, scale=\vertexscale] (b\i) at (\rightX,0) {};
   \draw [unset2factorstyle] (a\i) -- ++(90:0.7);
   \draw [unset2factorstyle] (b\i) -- ++(90:0.7);
   \draw [unset2factorstyle] (a\i) to (b\i);
}
\ifnum\blockparameter>1
\foreach \i in {1,...,\blockparameterminus}{%
   \pgfmathtruncatemacro\iNext{\i+1}
   \draw [sigma1style] (b\i) to (a\iNext);
}
\fi
\draw [connectorstyle] (a1) -- ++(180:0.5);
\draw [connectorstyle] (b\blockparameter) -- ++(0:0.5);

\end{tikzpicture}

%% file: DLBW.tex
\ifx \tikzscale \undefined
   \def \tikzscale {1}
\fi%
\ifx \blockparameter \undefined
   \def \blockparameter {2}%
\fi%
\tikzset{%
  sigma1style/.prefix style={thick,black},%
  connectorstyle/.prefix style={thick,gray!50},%
  unset2factorstyle/.prefix style={thick,blue,densely dashed}%
}
\begin{tikzpicture}[scale=\tikzscale,baseline=-2]

\pgfmathtruncatemacro\blockparameterminus{\blockparameter-1}
\pgfmathsetmacro\vertexscale{0.5*\tikzscale}
\pgfmathsetmacro\bbRightX{2*\blockparameter}

\useasboundingbox (-0.5,-0.5) rectangle (\bbRightX,1.5);

\foreach \i in {1,...,\blockparameter}{%
   \pgfmathsetmacro\leftX{(\i-1)*2}
   \pgfmathsetmacro\rightX{(\i-1)*2+1}
   \node [fill,circle, scale=\vertexscale] (a\i) at (\leftX,0) {};
   \node [fill,circle, scale=\vertexscale] (b\i) at (\rightX,0) {};
   \draw [unset2factorstyle] (a\i) -- ++(90:0.7);
   \draw [unset2factorstyle] (b\i) -- ++(90:0.7);
   \draw [unset2factorstyle] (a\i) to (b\i);
}
\ifnum\blockparameter>1
\foreach \i in {1,...,\blockparameterminus}{%
   \pgfmathtruncatemacro\iNext{\i+1}
   \draw [sigma1style] (b\i) to (a\iNext);
}
\fi
\draw [sigma1style] (a1) -- ++(180:0.5);
\draw [sigma1style] (b\blockparameter) -- ++(0:0.5);

\end{tikzpicture}

%% file: LBW.tex
\ifx \tikzscale \undefined
   \def \tikzscale {1}
\fi%
\ifx \blockparameter \undefined
   \def \blockparameter {2}%
\fi%
\tikzset{%
  sigma1style/.prefix style={thick,black},%
  connectorstyle/.prefix style={thick,gray!50},%
  unset2factorstyle/.prefix style={thick,blue,densely dashed}%
}
\begin{tikzpicture}[scale=\tikzscale,baseline=-2]

\pgfmathtruncatemacro\blockparameterminus{\blockparameter-1}
\pgfmathsetmacro\vertexscale{0.5*\tikzscale}
\pgfmathsetmacro\bbRightX{2*\blockparameter}

\useasboundingbox (-0.5,-0.5) rectangle (\bbRightX,1.5);

\foreach \i in {1,...,\blockparameter}{%
   \pgfmathsetmacro\leftX{(\i-1)*2}
   \pgfmathsetmacro\rightX{(\i-1)*2+1}
   \node [fill,circle, scale=\vertexscale] (a\i) at (\leftX,0) {};
   \node [fill,circle, scale=\vertexscale] (b\i) at (\rightX,0) {};
   \draw [unset2factorstyle] (a\i) -- ++(90:0.7);
   \draw [unset2factorstyle] (b\i) -- ++(90:0.7);
   \draw [unset2factorstyle] (a\i) to (b\i);
}
\ifnum\blockparameter>1
\foreach \i in {1,...,\blockparameterminus}{%
   \pgfmathtruncatemacro\iNext{\i+1}
   \draw [sigma1style] (b\i) to (a\iNext);
}
\fi
\draw [connectorstyle] (a1) -- ++(180:0.5);
\draw [sigma1style] (b\blockparameter) -- ++(0:0.5);
\end{tikzpicture}

%% file: BWN.tex
\ifx \tikzscale \undefined
   \def \tikzscale {1}
\fi%
\ifx \blockparameter \undefined
   \def \blockparameter {2}%
\fi%
\tikzset{%
  sigma1style/.prefix style={thick,black},%
  connectorstyle/.prefix style={thick,gray!50},%
  unset2factorstyle/.prefix style={thick,blue,densely dashed}%
}
\begin{tikzpicture}[scale=\tikzscale,baseline=-2]

\pgfmathsetmacro\vertexscale{0.5*\tikzscale}
\pgfmathsetmacro\necklaceRadius{0.5/sin(90/\blockparameter)}
\pgfmathsetmacro\necklaceRadiusLarge{\necklaceRadius+0.7}
\pgfmathsetmacro\necklaceAngle{180/\blockparameter}

%\useasboundingbox (-0.5,-0.5) rectangle (0.5,0.5);

\foreach \i in {1,...,\blockparameter}{%
   \pgfmathsetmacro\currentAngleA{(2*(\i-1))*\necklaceAngle}
   \pgfmathsetmacro\currentAngleB{(2*(\i-1)+1)*\necklaceAngle}
   \node [fill,circle, scale=\vertexscale] (a\i) at (\currentAngleA:\necklaceRadius) {};
   \node [fill,circle, scale=\vertexscale] (b\i) at (\currentAngleB:\necklaceRadius) {};
   \draw [unset2factorstyle] (a\i) to (b\i);
   \draw [unset2factorstyle] (a\i) to (\currentAngleA:\necklaceRadiusLarge);
   \draw [unset2factorstyle] (b\i) to (\currentAngleB:\necklaceRadiusLarge);
   \ifnum\i=\blockparameter
      \pgfmathsetmacro\markAngle{-\necklaceAngle/2}
      \draw [sigma1style] (a1) -- (b\blockparameter);
   \fi
   \ifnum\i>1
      \pgfmathtruncatemacro\iPrev{\i-1}
      \pgfmathsetmacro\markAngle{(2*\i-2.5)*\necklaceAngle}
      \draw [sigma1style] (a\i) -- (b\iPrev);
   \fi
}
\end{tikzpicture}

%% file: Q4.tex
\ifx \tikzscale \undefined
   \def \tikzscale {1}
\fi%
\tikzset{%
  sigma1style/.prefix style={thick,black},%
  connectorstyle/.prefix style={thick,gray!50},%
  unset2factorstyle/.prefix style={thick,blue,densely dashed}%
}
\begin{tikzpicture}[scale=\tikzscale,baseline=-2]

\pgfmathsetmacro\vertexscale{0.5*\tikzscale}

\useasboundingbox (-0.5,-0.5) rectangle (0.5,0.5);

\node [fill,circle, scale=\vertexscale] (v) at (0,0) {};
\draw [unset2factorstyle,bend left] (v) -- (45:0.7);
\draw [unset2factorstyle,bend right] (v) -- (135:0.7);
\draw [connectorstyle] (v) -- ++(270:0.7);
\end{tikzpicture}

%% file: T.tex
\ifx \tikzscale \undefined
   \def \tikzscale {1}
\fi%
\tikzset{%
  sigma1style/.prefix style={thick,black},%
  connectorstyle/.prefix style={thick,gray!50},%
  unset2factorstyle/.prefix style={thick,blue,densely dashed}%
}
\begin{tikzpicture}[scale=\tikzscale,baseline=-2]

\pgfmathsetmacro\vertexscale{0.5*\tikzscale}

\useasboundingbox (-0.5,-0.5) rectangle (0.5,0.5);

\node [fill,circle, scale=\vertexscale] (v) at (0,0) {};
\draw [unset2factorstyle,bend left] (v) -- (45:0.7);
\draw [unset2factorstyle,bend right] (v) -- (135:0.7);
\draw [sigma1style] (v) -- ++(270:0.7);
\end{tikzpicture}

%% file: not_closed_partial_pregraph.tex
\ifx \tikzscale \undefined
   \def \tikzscale {1}
\fi%
\tikzset{%
  connectorstyle/.prefix style={thick,gray!50},%
  unset2factorstyle/.prefix style={thick,blue,densely dashed}%
}
\begin{tikzpicture}[scale=\tikzscale,baseline=-2]

\pgfmathsetmacro\vertexscale{0.5*\tikzscale}

%\useasboundingbox (-0.5,-0.5) rectangle (1.5,1.5);

\node [fill,circle, scale=\vertexscale] (s1) at (45:0.5) {};
\node [fill,circle, scale=\vertexscale] (s2) at (135:0.5) {};
\node [fill,circle, scale=\vertexscale] (s3) at (225:0.5) {};
\node [fill,circle, scale=\vertexscale] (s4) at (315:0.5) {};

\draw [unset2factorstyle] (s1) to (s2);
\draw [unset2factorstyle] (s2) to (s3);
\draw [unset2factorstyle] (s3) to (s4);
\draw [unset2factorstyle] (s4) to (s1);

\node [fill,circle, scale=\vertexscale] (pc1) at (135:1.5) {};
\node [fill,circle, scale=\vertexscale] (pc2) at (135:2.5) {};
\draw [unset2factorstyle,bend left] (pc1) to (pc2);
\draw [unset2factorstyle,bend right] (pc1) to (pc2);
\draw [thick,black] (pc2)  -- ++(135:0.5);

\node [fill,circle, scale=\vertexscale] (q1) at (45:1.5) {};
\node [fill,circle, scale=\vertexscale] (q3) at (225:1.5) {};
\node [fill,circle, scale=\vertexscale] (q4) at (315:1.5) {};
   
\draw [unset2factorstyle] (q1) -- ++(0:0.5);
\draw [unset2factorstyle] (q1) -- ++(90:0.5);

\draw [unset2factorstyle] (q3) -- ++(180:0.5);
\draw [unset2factorstyle] (q3) -- ++(270:0.5);
\draw [thick,black] (q3)  -- (s3);

\draw [unset2factorstyle] (q4) -- ++(0:0.5);
\draw [unset2factorstyle] (q4) -- ++(270:0.5);
\end{tikzpicture}

%% file: not_closed_partial_pregraph_ext1.tex
\ifx \tikzscale \undefined
   \def \tikzscale {1}
\fi%
\tikzset{%
  connectorstyle/.prefix style={thick,gray!50},%
  unset2factorstyle/.prefix style={thick,blue,densely dashed}%
}
\begin{tikzpicture}[scale=\tikzscale,baseline=-2]

\pgfmathsetmacro\vertexscale{0.5*\tikzscale}

%\useasboundingbox (-0.5,-0.5) rectangle (1.5,1.5);

\node [fill,circle, scale=\vertexscale] (s1) at (45:0.5) {};
\node [fill,circle, scale=\vertexscale] (s2) at (135:0.5) {};
\node [fill,circle, scale=\vertexscale] (s3) at (225:0.5) {};
\node [fill,circle, scale=\vertexscale] (s4) at (315:0.5) {};

\draw [unset2factorstyle] (s1) to (s2);
\draw [unset2factorstyle] (s2) to (s3);
\draw [unset2factorstyle] (s3) to (s4);
\draw [unset2factorstyle] (s4) to (s1);

\node [fill,circle, scale=\vertexscale] (pc1) at (135:1.5) {};
\node [fill,circle, scale=\vertexscale] (pc2) at (135:2.5) {};
\draw [unset2factorstyle,bend left] (pc1) to (pc2);
\draw [unset2factorstyle,bend right] (pc1) to (pc2);
\draw [thick,black] (pc2)  -- ++(135:0.5);

\node [fill,circle, scale=\vertexscale] (q1) at (45:1.5) {};
\node [fill,circle, scale=\vertexscale] (q3) at (225:1.5) {};
\node [fill,circle, scale=\vertexscale] (q4) at (315:1.5) {};
   
\draw [unset2factorstyle] (q1) -- ++(0:0.5);
\draw [unset2factorstyle] (q1) -- ++(90:0.5);

\draw [unset2factorstyle] (q3) -- ++(180:0.5);
\draw [unset2factorstyle] (q3) -- ++(270:0.5);
\draw [thick,black] (q3)  -- (s3);

\draw [unset2factorstyle] (q4) -- ++(0:0.5);
\draw [unset2factorstyle] (q4) -- ++(270:0.5);
\draw [thick,black] (q4)  -- (s4);

\draw [thick,black] (pc1)  -- (s2);
\end{tikzpicture}

%% file: not_closed_partial_pregraph_ext2.tex
\ifx \tikzscale \undefined
   \def \tikzscale {1}
\fi%
\tikzset{%
  connectorstyle/.prefix style={thick,gray!50},%
  unset2factorstyle/.prefix style={thick,blue,densely dashed}%
}
\begin{tikzpicture}[scale=\tikzscale,baseline=-2]

\pgfmathsetmacro\vertexscale{0.5*\tikzscale}

%\useasboundingbox (-0.5,-0.5) rectangle (1.5,1.5);

\node [fill,circle, scale=\vertexscale] (s1) at (45:0.5) {};
\node [fill,circle, scale=\vertexscale] (s2) at (135:0.5) {};
\node [fill,circle, scale=\vertexscale] (s3) at (225:0.5) {};
\node [fill,circle, scale=\vertexscale] (s4) at (315:0.5) {};

\draw [unset2factorstyle] (s1) to (s2);
\draw [unset2factorstyle] (s2) to (s3);
\draw [unset2factorstyle] (s3) to (s4);
\draw [unset2factorstyle] (s4) to (s1);

\node [fill,circle, scale=\vertexscale] (pc1) at (135:1.5) {};
\node [fill,circle, scale=\vertexscale] (pc2) at (135:2.5) {};
\draw [unset2factorstyle,bend left] (pc1) to (pc2);
\draw [unset2factorstyle,bend right] (pc1) to (pc2);
\draw [thick,black] (pc2)  -- ++(135:0.5);

\node [fill,circle, scale=\vertexscale] (q1) at (45:1.5) {};
\node [fill,circle, scale=\vertexscale] (q3) at (225:1.5) {};
\node [fill,circle, scale=\vertexscale] (q4) at (315:1.5) {};
   
\draw [unset2factorstyle] (q1) -- ++(0:0.5);
\draw [unset2factorstyle] (q1) -- ++(90:0.5);

\draw [unset2factorstyle] (q3) -- ++(180:0.5);
\draw [unset2factorstyle] (q3) -- ++(270:0.5);
\draw [thick,black] (q3)  -- (s3);

\draw [unset2factorstyle] (q4) -- ++(0:0.5);
\draw [unset2factorstyle] (q4) -- ++(270:0.5);
\draw [thick,black] (q4)  -- (s4);

\draw [thick,black] (pc1)  -- (s1);
\end{tikzpicture}

%% file: table_overview_lists-marked-unmarked_count_time_approx.tex
%twiadria.ugent.be:/home/nvcleemp/Work/ddgraphs/results/makeTable.py

1 & 1 & 0.0s & 1 & 0.0s & 1 & 0.0s & 0.0s \\
2 & 5 & 0.0s & 5 & 0.0s & 3 & 0.0s & 0.0s \\
3 & 2 & 0.0s & 2 & 0.0s & 2 & 0.0s & 0.0s \\
4 & 13 & 0.0s & 13 & 0.0s & 9 & 0.0s & 0.0s \\
5 & 7 & 0.0s & 7 & 0.0s & 7 & 0.0s & 0.0s \\
6 & 31 & 0.0s & 31 & 0.0s & 29 & 0.0s & 0.0s \\
7 & 25 & 0.0s & 27 & 0.0s & 27 & 0.0s & 0.0s \\
8 & 103 & 0.0s & 109 & 0.0s & 105 & 0.0s & 0.0s \\
9 & 86 & 0.0s & 118 & 0.0s & 118 & 0.0s & 0.0s \\
10 & 311 & 0.0s & 394 & 0.0s & 392 & 0.0s & 0.1s \\
11 & 260 & 0.0s & 546 & 0.0s & 546 & 0.0s & 0.3s \\
12 & 938 & 0.0s & 1 726 & 0.0s & 1 722 & 0.1s & 1.3s \\
13 & 763 & 0.0s & 2 701 & 0.1s & 2 701 & 0.1s & 5.2s \\
14 & 2 521 & 0.0s & 7 955 & 0.3s & 7 953 & 0.3s & 22.0s \\
15 & 1 968 & 0.0s & 13 966 & 0.4s & 13 966 & 0.4s & 94.8s \\
16 & 6 776 & 0.0s & 40 039 & 1.4s & 40 035 & 1.5s & $\approx$ 7.0m \\
17 & 5 171 & 0.0s & 75 341 & 2.3s & 75 341 & 2.2s & $\approx$ 31.7m \\
18 & 16 557 & 0.0s & 210 765 & 8.1s & 210 763 & 8.0s & $\approx$ 2.5h \\
19 & 12 321 & 0.0s & 420 422 & 13.9s & 420 422 & 14.0s & $\approx$ 11.6h \\
20 & 40 622 & 0.1s & 1 162 196 & 46.5s & 1 162 192 & 46.6s & $\approx$ 56.0h \\
21 & 29 843 & 0.1s & 2 419 060 & 86.8s & 2 419 060 & 86.7s \\
22 & 93 166 & 0.2s & 6 626 610 & $\approx$ 4.6m & 6 626 608 & $\approx$ 4.6m \\
23 & 67 345 & 0.2s & 14 292 180 & $\approx$ 9.2m & 14 292 180 & $\approx$ 9.2m \\
24 & 213 822 & 0.5s & 38 958 571 & $\approx$ 28.2m & 38 958 567 & $\approx$ 28.4m \\
25 & 153 388 & 0.5s & 86 488 183 & $\approx$ 59.7m & 86 488 183 & $\approx$ 59.8m \\
26 & 467 050 & 1.2s & 235 004 260 & $\approx$ 3.0h & 235 004 258 & $\approx$ 3.0h \\
27 & 331 411 & 1.2s & 534 796 010 & $\approx$ 6.6h & 534 796 010 & $\approx$ 6.6h \\
28 & 1 018 009 & 3.0s & 1 450 990 715 & $\approx$ 19.2h & 1 450 990 711 & $\approx$ 19.2h \\
29 & 719 250 & 2.9s & 3 373 088 492 & $\approx$ 43.7h & 3 373 088 492 & $\approx$ 43.7h \\
30 & 2 136 996 & 6.8s & 9 147 869 420 & $\approx$ 5.3 days & 9 147 869 418 & $\approx$ 5.3 days \\
31 & 1 498 823 & 6.5s & 21 667 784 040 & $\approx$ 12.5 days & 21 667 784 040 & $\approx$ 12.5 days \\
32 & 4 468 381 & 16.1s & 58 791 212 723 & $\approx$ 36.0 days & 58 791 212 719 & $\approx$ 36.1 days \\
33 & 3 126 211 & 15.4s & 141 583 919 924 & $\approx$ 86.2 days & 141 583 919 924 & $\approx$ 86.3 days \\
34 & 9 071 315 & 34.5s & 384 597 958 574 & $\approx$ 246.2 days & 384 597 958 572 & $\approx$ 246.4 days\\
35 & 6 316 138 & 33.1s & 940 092 232 951 & $\approx$ 600.6 days & 940 092 232 951 & $\approx$ 601.1 days \\

%% file: table_ddgraphs_approx.tex
%twiadria.ugent.be:/home/nvcleemp/Work/ddgraphs/results/process_ddgraphs.py
\begin{tabular}{rrrr}
\toprule
\cth{$n$} & \cth{Delaney-Dress graphs} & \cth{time} & \cth{rate} \\
\midrule
1 & 1 & 0.0s &  \\
2 & 7 & 0.0s &  \\
3 & 3 & 0.0s &  \\
4 & 22 & 0.0s &  \\
5 & 13 & 0.0s &  \\
6 & 70 & 0.0s &  \\
7 & 67 & 0.0s &  \\
8 & 315 & 0.0s &  \\
9 & 393 & 0.0s &  \\
10 & 1 577 & 0.0s &  \\
11 & 2 515 & 0.0s &  \\
12 & 9 480 & 0.1s & 94 800.00/s \\
13 & 17 205 & 0.1s & 172 050.00/s \\
14 & 61 594 & 0.3s & 205 313.33/s \\
15 & 123 953 & 0.4s & 309 882.50/s \\
16 & 433 030 & 1.6s & 270 643.75/s \\
17 & 931 729 & 2.5s & 372 691.60/s \\
18 & 3 196 841 & 9.1s & 351 301.21/s \\
19 & 7 258 011 & 16.3s & 445 276.75/s \\
20 & 24 630 262 & 55.0s & 447 822.95/s \\
21 & 58 309 071 & 105.9s & 550 605.01/s \\
22 & 196 266 434 & $\approx$ 5.8m & 568 064.93/s \\
23 & 481 330 615 & $\approx$ 12.0m & 666 478.28/s \\
24 & 1 610 942 856 & $\approx$ 38.8m & 691 629.25/s \\
25 & 4 071 117 829 & $\approx$ 1.4h & 785 187.34/s \\
26 & 13 569 014 653 & $\approx$ 4.6h & 826 265.50/s \\
27 & 35 202 390 477 & $\approx$ 10.6h & 919 758.85/s \\
28 & 116 994 675 348 & $\approx$ 33.8h & 960 576.60/s \\
29 & 310 624 700 725 & $\approx$ 3.4 days & 1 049 801.45/s \\
30 & 1 030 455 432 427 & $\approx$ 11.0 days & 1 084 892.06/s \\
31 & 2 792 944 867 743 & $\approx$ 27.4 days & 1 177 978.85/s \\
32 & 9 256 071 637 206 & $\approx$ 88.8 days & 1 205 812.64/s \\
33 & 25 557 439 215 047 & $\approx$ 231.9 days & 1 275 360.82/s \\
34 & 84 676 700 443 607 & $\approx$ 2.1 years & 1 297 545.26/s\\
35 & 237 766 612 990 437 & $\approx$ 5.6 years & 1 353 707.99/s\\
\bottomrule
\end{tabular}

%% file: ddgraphs.bbl
\begin{thebibliography}{1}

\bibitem{OT:67}
B.~Owens and H.M. Trent.
\newblock On determining minimal singularities for the realization of an
  incidence sequence.
\newblock {\em SIAM Journal on Applied Mathematics}, 15:406--418, 1967.

\bibitem{McK81__2}
B.~D. McKay.
\newblock Practical graph isomorphism.
\newblock {\em Congressus Numerantium}, 30:45--87, 1981.

\bibitem{DressScharlau86}
A.~Dress and R.~Scharlau.
\newblock The 37 combinatorial types of minimal non-transitive, equivariant
  tilings of the euclidean plane.
\newblock {\em Discrete Mathematics}, 60:121--138, 1986.

\bibitem{DressHuson87}
A.~Dress and D.~Huson.
\newblock On tilings of the plane.
\newblock {\em Geometriae Dedicata}, 24:295--310, 1987.

\bibitem{DrBr96}
A.~Dress and G.~Brinkmann.
\newblock Phantasmagorical fulleroids.
\newblock {\em {{MATCH Commun. Math. Comput. Chem.}}}, 33:87--100, 1996.
\newblock {P.W.~Fowler~ed.}: Mathematical aspects of the fullerenes.

\bibitem{McK96}
B.~D. McKay.
\newblock Isomorph-free exhaustive generation.
\newblock {\em Journal of Algorithms}, 26:306--324, 1998.

\bibitem{Br09}
G. Brinkmann.
\newblock Generating water clusters and other directed graphs.
\newblock {\em J. Math. Chem.}, 46:1112--1121, 2009.

\bibitem{azul}
G.~Brinkmann, O.~Delgado-Friedrichs, E.~C. Kirby, and N.~Van~Cleemput.
\newblock {A Catalogue of Periodic Fully-Resonant Azulene-Transitive Azulenoid
  Tilings Analogous to Clar Structures}.
\newblock {\em {Croatica Chemica Acta}}, {82}({4}):{781--789}, {2009}.

\bibitem{Br:13}
G. Brinkmann.
\newblock Generating regular directed graphs.
\newblock {\em Discrete Mathematics}, 313(1):1--7, 2013.

\bibitem{pregraphs}
G. Brinkmann, N. Van~Cleemput, and T. Pisanski.
\newblock Generation of various classes of trivalent graphs.
\newblock {\em Theoretical Computer Science}, 502:16--29, 2013.

\bibitem{HOPF:13}
I. Hubard, A. Orbani\'c, T. Pisanski, and M. del
  R\'io~Francos.
\newblock Medial symmetry type graphs.
\newblock {\em Electronic J. Combin.}, 20(3), \#P29, 2013.

\bibitem{ddgraphssite}
ddgraphs website.
\newblock http://caagt.ugent.be/ddgraphs.

\end{thebibliography}
